\documentclass[12pt]{amsart}

\usepackage{amsmath,amsthm,amssymb}
\usepackage[colorlinks=true]{hyperref}

\usepackage{graphicx}

\newcommand{\R}{\mathbb{R}}
\newcommand{\N}{\mathbb{N}}
\newcommand{\C}{\mathbb{C}}
\newcommand{\Z}{\mathbb{Z}}

\DeclareMathOperator{\ran}{Ran}

\DeclareMathOperator{\Op}{Op}
\DeclareMathOperator{\Span}{span}

\newtheorem{definition}{Definition}
\newtheorem{proposition}{Proposition}
\newtheorem{corollary}{Corollary}
\newtheorem{lemma}{Lemma}
\newtheorem{theorem}{Theorem}

\begin{document}

\begin{abstract}
The purpose of this note is to study a class of operators generalizing Bergman kernels. Among other results, we obtain a (micro-)local description of a class of FIOs that are elliptic projectors, in the analytic category. 
\end{abstract}

\title{Microlocal Projectors}

\author{Yannick Guedes Bonthonneau}

\date{October 2022}

\maketitle

In november 2022, mathematicians gathered at CIRM under the theme of ``Bergman kernels in microlocal analysis and mathematical physics''. The main object of this conference was the reproducing kernel introduced by Bergman \cite{Bergman-1932}, associated with the space of holomorphic functions of a given domain. This central object in analysis of several complex variables has many applications in the classification of said domains, since it is a holomorphic invariant. Its study motivated the development of many tools of microlocal analysis, for example the works of Fefferman \cite{Fefferman-1974} or Boutet-Sj\"ostrand \cite{Boutet-Sjostrand-76}. Ninety years after its discovery, the conference showed how this is still a very active domain.

This theme comes in two different flavors: on the one hand the case of open pseudo-convex domains in complex manifolds, related to the Szeg\"o kernel. On the other hand the study of the space of holomorphic sections of complex line bundles over compact closed manifolds. In both cases, a lot of effort in the 1970's went into proving that these Bergman kernels behave as FIOs with complex phase. To explain what this means, we have to choose one side, and we will in this article consider the ``line bundle'' case. To this end, consider $L\to M$ a holomorphic complex line bundle over a compact complex manifold. Provided some geometric condition on $L$ is satisfied, the $L^2$ orthogonal projection on the space of holomorphic sections of $L^k$ has its kernel satisfying (near the diagonal)
\begin{equation}\label{eq:FIO-complex-phase}
B_k(\alpha,\beta) = e^{ i k \phi_L(\alpha,\beta)} a_k(\alpha,\beta) + \mathcal{O}_{k\to+\infty}(k^{-\infty}). 
\end{equation}
When suitably normalized the function $\phi_L$ appearing is Calabi's Diastasis function \cite{Calabi-1953}. It is complex valued, with non negative imaginary part, hence the name ``FIO with complex phase''. The function $a_k$ is a smooth amplitude.  

While they proved to be central for the geometry of complex domains, the Bergman kernels are also of interest in the context of mathematical physics, particularly for questions of quantization. While K\"ahler manifolds are naturally symplectic manifolds, the converse is not true. Victor Guillemin and Alejandro Uribe \cite{Guillemin-Uribe-88} discovered that with an almost complex structure, one could produce, on any symplectic manifold, a vector space (of eigenfunctions) that somehow behaved like the space of holomorphic sections in the K\"ahler case.  Followed a period of intense activity, motivated by the Guillemin-Sternberg question ``does quantization commute with reduction''. Borthwick and Uribe \cite{Borthwick-Uribe-1996} were the ones to provide a (more or less) complete setup, with Tœplitz-like operators. There are thus two versions of the theory, one with $\mathrm{Spin}^c$ Dirac operators, and the other one with a generalization of the Hodge Laplacian. 

In any case, in these new schemes of quantization, an asymptotic similar to that of \eqref{eq:FIO-complex-phase} is valid. In the fall of 2022, I presented a study of elliptic FIOs with complex phase (i.e operators whose kernels satisfy an expansion in the form \eqref{eq:FIO-complex-phase}) that are projectors. The purpose of this note is to be a recording of these investigations.

While dealing with the space of holomorphic sections, to build a Bergman kernel, one needs a connection, which may be only $C^\infty$ (or even $C^k$ !). For quite some time, most of the theory was done in this case. However recently, in particular with the work of Sj\"ostrand, V\~u Ng\d{o}c and Rouby \cite{RSV-18}, appeared the possibility to upgrade usual $C^\infty$ results with $\mathcal{O}(h^\infty)$ remainders to analytic results with $\mathcal{O}(e^{-c/h})$ errors. In this paper, we will consider mainly the analytic setting. It is however the case that replacing holomorphic extensions with almost analytic extensions\footnote{in the spirit of \cite{Melin-Sjostrand-75}.} whenever relevant would provide almost complete proofs for most statements for the $C^\infty$ case.

Let us start with a definition:
\begin{definition}
Let $U\subset \R^m$ be an open set, and $\phi: U\times U \to \C$ (potentially only defined in a neighbourhood of the diagonal) be a real analytic function satisfying
\begin{enumerate}
	\item $\phi(\alpha,\alpha) = 0$
	\item $\Im \phi(\alpha,\beta) \geq |\alpha-\beta|^2 / C$ for some $C>0$
	\item For $\alpha,\beta$ close enough, $\gamma \mapsto \phi(\alpha,\gamma)+\phi(\gamma,\beta)$ has exactly one critical point $\gamma_c(\alpha,\beta)$ close to $\alpha,\beta$, and 
	\begin{equation}\label{eq:reproducing}
	\phi(\alpha,\beta) = \phi(\alpha,\gamma_c(\alpha,\beta)) + \phi(\gamma_c(\alpha,\beta) , \beta). 
	\end{equation}
\end{enumerate}
Then we say that $\phi$ is a \emph{projector phase} on $U$. If additionally we have
\[
\phi(\alpha,\beta) = - \overline{\phi(\beta,\alpha)},
\]
we say that $\phi$ is \emph{self-adjoint projector phase}. 
\end{definition}

Notice that item (2) implies the existence and uniqueness of item (3) according to Implicit Functions. What is really crucial is the \emph{reproducing property} equation \eqref{eq:reproducing}.

This result was presented at CIRM in 2022:
\begin{theorem}\label{thm:main}
Let $\phi$ be a projector phase on $U$. Then $U$ is endowed with a triple $(g,J,\omega)$, where $\omega$ is an exact symplectic form, $g$ is a $2$-symmetric (complex valued) non-degenerate tensor, and $J$ an endomorphism, so that for $u,v\in TU$,
\[
g(Ju,v)=\omega(u,v).
\]
If $\phi$ is self-adjoint, $g$ is a metric. Additionally, we can find an amplitude $a(h;\alpha,\beta)$ analytic and supported near the diagonal such that the operator $\Pi$ on $C^\infty(U)$ with kernel
\[
e^{\frac{i}{h} \phi(\alpha,\beta)} a(h;\alpha,\beta)
\]
satisfies (for $0<h$ sufficiently small) 
\begin{equation}\label{eq:Def-approximate-projector}
\Pi^2 f = \Pi f + \mathcal{O}(e^{-c/h}),\ f\in C^\infty_c(U). 
\end{equation}
\end{theorem}
In the usual constructions of Bergman or Bergman-like kernels, one usually starts from such an almost K\"ahler triple as above, and the gist of this theorem is that if there exists a Bergman-like kernel, there must also exist such an underlying geometrical datum.

In a second part of the paper, we will prove the following
\begin{theorem}\label{thm:conjugation}
Assume that $U$ is a ball, and let $\Pi_1, \Pi_2$ be two operators as in Theorem \ref{thm:main}. Then $\Pi_1$ and $\Pi_2$ can be microlocally conjugated by $U$ an elliptic FIO with complex phase. Additionally, if they are both self-adjoint, $U$ can be assumed to be unitary, with real phase. 
\end{theorem}
For more precision on the terms of the statement, see Theorem \ref{thm:local-structure-projectors} and Corollary \ref{cor:decomp-self-adjoint}. This result immediately raises the question of whether one can obtain global conjugation results in the presence of topology. The study of this interesting problem is left to future investigations. 

The author is indebted to Laurent Charles, San V\~u Ng\d{o}c, Johannes Sj\"ostrand and Alix Deleporte for insightful conversations. It is also a pleasure to thank again the organizers (Deleporte, Sj\"ostrand, and Michael Hitrik) and CIRM for this beautiful conference. Finally, the author acknowledges the support of Agence Nationale de la Recherche through
the PRC grant ADYCT (ANR-20-CE40-0017).

\section{November 2022 in CIRM}
\label{part:1}
Here we give detailed statements and proofs corresponding to the presentation in Luminy. 

\subsection{Study of Projector phases}
\label{sec:study-phase}

In this section, we will study projector phases, and show that they can be factorized. This is closely related to \S (3.4) in \cite{Boutet-Guillemin}, and in particular the enigmatic sentence
\begin{quote}
Now $S$ and $HH^*$ are both adapted to $Id_\Sigma$, of degree 0. So they are also Hermite operators with respect to $\Sigma$, whose symbols are projectors hence of rank 1 (cf. (2.11)). Since $S$ and $HH^*$ are both orthogonal and $S\circ HH^* \sim HH^*$, their symbols are equal
\end{quote} 

A good part of our arguments are quite elementary but rely on manipulation of partial derivatives, so we fix notational conventions. Let $f$ be a function depending on $\alpha,\beta\in U$ (or some open set of $U\times U$). By $\partial_\alpha f$, we will denote the row vector
\[
\left(\frac{\partial f}{\partial \alpha_1}, \dots, \frac{\partial f}{\partial \alpha_m} \right).
\]
(and likewise for $\partial_\beta f$). We will identify $\partial_\alpha f$ with the $1$-form $\partial_\alpha f \cdot d\alpha$. We will also denote by $\partial_\alpha\partial_\beta f$ the $m\times m$ matrix with coefficients
\[
\begin{pmatrix}
\frac{\partial^2 f}{\partial \alpha_1 \partial \beta_1} & \dots & \frac{\partial^2 f}{\partial\alpha_1\partial\beta_m} \\
 \dots & \dots & \dots \\
\frac{\partial^2 f}{\partial\alpha_m \partial \beta_1} &\dots &\frac{\partial^2 f}{\partial\alpha_m \partial \beta_m} 
\end{pmatrix}.
\]
With this convention, $(\partial_\alpha\partial_\beta f)^\top = \partial_\beta\partial_\alpha f$.

\subsubsection{Quadratic behaviour near the diagonal}

In order to prove Theorem \ref{thm:main}, we will start by studying the (local) properties of the phase. Let us assume that $U$ is simply connected. Then we obtain
\begin{proposition}\label{prop:geometric-objects}
The dimension $m=2n$ is even, and $U$ is endowed with a compatible triple $(\omega,g,J)$, where $\omega=d\vartheta$ is an exact real symplectic 2-form, $g$ is a symmetric (complex with positive real part) non-degenerate 2-form, and $J$ is an endomorphism of $TU\otimes \C$, so that
\[
J^2=-1,\ g(Ju,Jv) =g(u,v) \text{ and }g(Ju,v)=\omega(u,v).
\]
If $\phi$ is self-adjoint, $g$ is also real, so that $(\omega,g,J)$ is an almost K\"ahler structure. Finally, there exists a real valued real analytic function $f$ on $U$ so that
\begin{equation}\label{eq:Taylor-the-phase}
\begin{split}
\phi(\alpha+u,\alpha+v)&-f(\alpha+u)+f(\alpha+v) = \\
		&\vartheta_\alpha(u-v) + \frac{i}{2}g_\alpha(u-v)  + \omega_\alpha(u,v) + \mathcal{O}(u,v)^3.
\end{split}
\end{equation}
\end{proposition}
Observe that $U$ is always symplectic in the usual sense, while the 2-tensor $g$ may be complex valued when $\phi$ is not self-adjoint. The rest of this section is devoted to proving Proposition \ref{prop:geometric-objects}. We start by introducing some notations:
\begin{equation}\label{eq:notations}
\vartheta:\alpha\mapsto \vartheta_\alpha = \partial_\alpha \phi_{|\beta=\alpha},\quad \omega = d\vartheta. 
\end{equation}
From the assumptions, we know that $\vartheta$ is a real valued $1$-form, so that $\omega$ is a closed real $2$-form, but we do not know a priori that it is non-degenerate.

Let us also denote (where the derivatives are taken at $\beta =\alpha$)
\[
A = \partial_\alpha^2\phi,\ B = \partial_\alpha\partial_\beta \phi,\ C = \partial_\beta^2\phi. 
\]
By definition, we have 
\begin{equation*}
\begin{split}
\phi(\alpha+u,\alpha+v) &= \partial_\alpha\phi\cdot u + \partial_\beta\phi\cdot v \\
		& + \frac{1}{2}\left( u^\top\cdot A \cdot u + 2 u^\top \cdot B \cdot v + v^\top \cdot C \cdot v\right) + \mathcal{O}(u,v)^3.
\end{split}
\end{equation*}
Differentiating the condition that $\phi(\alpha,\alpha) =0$, once and twice, we get that at $\beta=\alpha$,
\begin{equation}\label{eq:0-on-diagonal}
\partial_\alpha \phi = - \partial_\beta \phi,\ A+ B + B^\top + C = 0. 
\end{equation}
It is convenient to reorganize the expansion in terms of the matrices
\[
R = \frac{B^\top - B}{2},\, D =i\frac{ B + B^\top}{2},\, P= \frac{A-C}{2}.
\]
Constraint \eqref{eq:0-on-diagonal} rewrites as
\[
A+C = 2 i D.
\]
We can also write
\[
\vartheta_{\alpha+u} = \vartheta_\alpha + u^\top\cdot( A+ B^\top) + \mathcal{O}(u^2).  
\]
Since $\vartheta$ is real valued, we find that $A+ B^\top = A - iD + R = P + R$ is real. By definition, $P$ is symmetric, and $R$ is antisymmetric, so that both have to be real. Additionally, $R$ is the matrix of $\omega$, i.e
\[
\omega_\alpha(u,v) = u^\top \cdot R \cdot v.
\]
Since $P$ is a real symmetric matrix, and $U$ is simply connected, we can find a real valued primitive function $f$ so that $d^2 f = P$, and consider the new phase
\[
\phi_f (\alpha,\beta) = \phi(\alpha,\beta) - f(\alpha) + f(\beta). 
\]
The new phase satisfies the same assumptions as before, with $\vartheta$ replaced by $\vartheta - df$. However, now $P=0$, and $R$, $D$ are unchanged. We will make this assumption from now on.

Also, from the assumptions, we know that $D$ is symmetric with positive real part, so admits a symmetric square root $L$ -- with the determination that $L=1$ when $ D= 1$. Let us set
\[
g_\alpha(u,v)= u^\top \cdot D \cdot v,
\]
and
\[
J_\alpha u = - D^{-1} R u,\quad \tilde{J} = - L^{-1} R L^{-1}. 
\]
With these notation, we have ensured Formula \eqref{eq:Taylor-the-phase}. To complete the proof of Proposition \ref{prop:geometric-objects}, it remains to prove that $J^2=-1$, and that $\omega$ is non-degenerate. From the relation between $g$, $J$ and $\omega$, it suffices to prove that $J^2=-1$. For this we start with the observation:
\begin{lemma}
Let $\phi$ be a projector phase on $U$, and let $\alpha\in U$. Then the quadratic Taylor polynomial of $\phi$ at $(\alpha,\alpha)$ is a projector phase on $\R^m$. 
\end{lemma}

\begin{proof}
It suffices to prove item (3), the reproducing property. Let us denote $\phi^\alpha(u,v)$ the quadratic Taylor polynomial of $\phi$ at $(\alpha,\alpha)$, and $\gamma^\alpha_c(u,v)$ the critical point of 
\[
\gamma\mapsto \phi^\alpha(u,\gamma)+\phi^\alpha(\gamma,v).
\]
Certainly, $\gamma^\alpha_c$ satisfies
\[
\partial_v \phi^\alpha(u,\gamma^\alpha_c)+\partial_u\phi^\alpha(\gamma^\alpha_c,v) = 0. 
\]
This is a linear equation, and it is of the form $D \gamma^\alpha_c = \dots$, so that since $D$ is invertible, $(\beta,\delta)\mapsto \gamma^\alpha_c$ is uniquely defined and linear. On the other hand, we have
\begin{align*}
& 0 = \partial_\beta \phi(\alpha+u, \gamma_c(\alpha+u,\alpha+v)) + \partial_\alpha \phi(\gamma_c(\alpha+u,\alpha+v),\alpha+v)\\
	& = \partial_u \phi^\alpha(u,\gamma_c(\alpha+u,\alpha+v))+\partial_v\phi^\alpha(\gamma_c(\alpha+u,\alpha+v),v) + \mathcal{O}(u, v)^2,\\
	& = \partial_v \phi^\alpha(u,d_{(\alpha,\alpha)}\gamma_c(u,v))+\partial_v \phi^\alpha(d_{(\alpha,\alpha)}\gamma_c(u,v),v)+ \mathcal{O}(u,v)^2.
\end{align*}
We deduce that
\[
\gamma^\alpha_c(u,v) = d_{(\alpha,\alpha)}\gamma_c(u,v).
\]
Then,
\begin{align*}
\phi^\alpha(u,v) &= \phi(\alpha+u,\alpha+v) + \mathcal{O}( u,v)^3 \\
						&= \phi(\alpha+u,\gamma_c)+ \phi(\gamma_c ,\alpha+v) + \mathcal{O}(u,v)^3 \\
						&= \phi^\alpha(u,\gamma_c-\alpha)+ \phi^\alpha(\gamma_c-\alpha,v) + \mathcal{O}( u,v)^3 \\
						&= \phi^\alpha(u,\gamma_c^\alpha)+ \phi^\alpha(\gamma_c^\alpha ,v) + \mathcal{O}(u,v)^3
\end{align*}
(in the last line, we used that $\gamma_c^\alpha$ is the critical point).
\end{proof}

Using the computations and identities above, we can rewrite the taylor expansion as
\[
\psi^\alpha(u,v):=\phi^\alpha(L^{-1} u, L^{-1} v) = \vartheta_\alpha L^{-1}(u-v) + \frac{i}{2}(u-v)^2  + \langle \tilde{J}u, v \rangle. 
\]

\begin{proposition}\label{prop:quadratic-behaviour}
For $\alpha\in U$, 
\[
J^2 = -1. 
\]
In particular, $J$ (and $R$) are invertible, $\omega$ is symplectic, and $(\omega,g,J)$ is an almost K\"ahler triple (except that $g$ may be complex valued in the non-self-adjoint case). 
\end{proposition}
This closes the proof of Proposition \ref{prop:geometric-objects}.

\begin{proof}
For this we will use the reproducing property. Since it holds for $\phi^\alpha$, it also holds for $\psi^\alpha$. Let us find an expression for the critical point of
\[
w\mapsto \psi^\alpha(u,w)+\psi^\alpha(w,v).
\]
Direct computation gives
\[
w = \frac{u+v}{2} + \frac{i}{2}\tilde{J}(v-u), 
\]
so that the critical value is
\[
\vartheta_\alpha L^{-1}(u-v) + \frac{i}{4}(u-v)^2 + \frac{i}{4}(\tilde{J}(u-v))^2 + \frac{1}{2}\langle u+v,v-u\rangle.
\]
This should be equal to $\psi^\alpha(u,v)$. Choosing $v = -u$, this leads to 
\[
(\tilde{J}u)^2 = u^2, 
\]
whence $\tilde{J}^2 = -1$, because $\tilde{J}$ is anti-symmetric. Then $J^2 = L^{-1} \tilde{J}^2 L = -1$. 
\end{proof}

Since $\omega$ is symplectic, we deduce that $m=2n$, and we can find local Darboux coordinates $(x,\xi)$ near $\alpha$, so that 
\[
\vartheta_\beta = \xi dx, \quad \omega= dx\wedge d\xi. 
\]
In particular, this gives
\[
R = \begin{pmatrix}
 0 & - 1 \\ 1 & 0
\end{pmatrix},
\]

\begin{corollary}
The rank of $\partial_\alpha\partial_\beta \phi$ is $n$ on the diagonal.
\end{corollary}

\begin{proof}
From the computations above, we see that
\[
\partial_\alpha \partial_\beta \phi(\alpha,\alpha) = B =  - R - i D =  L \tilde{J} L - i L^2 = L(\tilde{J}-i)L. 
\]
But since $\tilde{J}^2 = -1$, $\tilde{J}$ is diagonalizable, and its eigenvalues are $\pm i$ with eigenspaces $E_{\pm i}$. Then because $\tilde{J}$ is anti-symmetric, $E_{\pm i}\subset E_{\pm i}^\perp$. By consideration of dimension, each $E_{\pm i}$ has dimension $n$, and we are done.
\end{proof}

\subsubsection{The involutive manifold}

Let us define subsets of $T^\ast U$:
\begin{align*}
\mathcal{J}_\phi &= \{ (\alpha, \partial_\alpha \phi(\alpha,\beta)) \ |\ \text{$\alpha$ and $\beta$ close to the diagonal in $U\times U$} \},\\
\mathcal{J}_\phi^\ast &= \{ (\beta, - \partial_\beta \phi(\alpha,\beta)) \ |\ \text{$\alpha$ and $\beta$ close to the diagonal in $U\times U$} \}.\\
\Sigma_\phi &= \{ (\alpha, \vartheta_\alpha)\ |\  \alpha\in U\}. 
\end{align*}
\begin{proposition}\label{prop:involutive}
The sets $\mathcal{J}_\phi$, $\mathcal{J}_\phi^\ast$ are germs of complex analytic submanifolds of $\C^{4n}$, of complex dimension $3n$, involutive for the standard symplectic form. Additionally, 
\[
\mathcal{J}_\phi \cap \mathcal{J}_\phi^\ast = \Sigma_\phi
\]
is a transversal intersection. If $\phi$ is self-ajoint, 
\[
\mathcal{J}_\phi^\ast = \overline{\mathcal{J}_\phi}.
\]
\end{proposition}

\subsubsection{The intersection property}\label{sec:associativity}

In this section, we prove
\begin{lemma}
\begin{equation}\label{eq:intersection of jphi}
\mathcal{J}_\phi \cap \mathcal{J}_\phi^\ast = \Sigma. 
\end{equation}
\end{lemma}

Let us introduce some notations. The phase $\phi$ defines a map $f$ from $\mathcal{J}_\phi^\ast$ to $\mathcal{P}( \mathcal{J}_\phi )$ (subsets of $\mathcal{J}_\phi$), given by
\[
f : (\beta, \xi) \mapsto \{ (\alpha, \partial_\alpha \phi(\alpha,\beta)) \ |\ \partial_\beta\phi(\alpha,\beta) = - \xi \}. 
\]
Now, from the definition of $\gamma_c=\gamma_c(\alpha,\beta)$, we have 
\begin{equation}\label{eq:d-critique}
\begin{split}
\partial_\beta \phi(\alpha,\gamma_c)	&= - \partial_\alpha\phi(\gamma_c,\beta) \\
\partial_\beta\phi(\gamma_c, \beta) 	&= \partial_\beta\phi(\alpha,\beta), \\ \partial_\alpha\phi(\alpha,\gamma_c) 	&= \partial_\alpha\phi(\alpha,\beta),
\end{split}
\end{equation}
so that 
\begin{equation*}
\begin{split}
(\alpha, \partial_\alpha\phi(\alpha,\gamma_c))&\in f(\gamma_c, \partial_\alpha\phi(\gamma_c,\beta))\\
(\gamma_c,\partial_\alpha\phi(\gamma_c,\beta))&\in f(\beta, - \partial_\beta \phi(\alpha,\beta)).
\end{split}
\end{equation*}
Let us denote
\[
\delta(\alpha,\partial_\alpha\phi(\alpha,\beta);\beta, - \partial_\beta\phi(\alpha,\beta)) = (\gamma_c,\  \partial_\alpha\phi(\gamma_c,\beta)).
\]
This defines a map $\delta$ from a subset of $\mathcal{J}_\phi \times \mathcal{J}_\phi^\ast$ to $\mathcal{J}_\phi\cap\mathcal{J}_\phi^\ast$. The identities above become
\begin{equation}\label{eq:identities-delta-f}
a\in f(b) \Longrightarrow \delta(a,b)\in f(b) \text{ and } a\in f(\delta(a,b)).
\end{equation}

Now, we observe that for $a\in f(b)$, $a$ and $b$ have the same first component if and only if $a=b\in \Sigma$. Indeed, $a\in f(b)$ means that there exist $\alpha,\beta$ such that
\[
a= (\alpha,\partial_\alpha\phi(\alpha,\beta))\text{ and } b=(\beta, -\partial_\beta\phi(\alpha,\beta)).
\]
Then $\alpha=\beta$ if and only if $a=b$, and then $a\in \Sigma$. On the other hand, it is easy to see that if $a\in\Sigma$, then $a\in f(a)$.

Now, if we have a point $c\in \mathcal{J}_\phi\cap\mathcal{J}_\phi^\ast$, it means that there exist $\gamma, \beta,\alpha,\gamma'$ such that
\[
c=(\gamma,\partial_\alpha \phi(\gamma,\beta))= (\gamma', -\partial_\beta\phi(\alpha,\gamma')),
\]
whence it comes that $\gamma=\gamma' = \gamma_c(\alpha,\beta)$, and $c=\delta(a,b)$ for some $a\in f(b)$. Hence the map $\delta$ is surjective.

We claim that
\begin{equation}\label{eq:associativity}
\gamma_c(\gamma_c(\alpha,\beta),\beta) = \gamma_c(\alpha,\gamma_c(\alpha,\beta)) = \gamma_c(\alpha,\beta). 
\end{equation}
If this is true, we are done. Indeed, it implies that if $a\in f(b)$, then for some $\alpha,\beta$,
\begin{align*}
\delta(a,\delta(a,b)) &= (\gamma_c(\alpha,\gamma_c(\alpha,\beta)), \partial_\alpha\phi(\gamma_c(\alpha,\gamma_c(\alpha,\beta)), \gamma_c(\alpha,\beta))),\\
					&= (\gamma_c(\alpha,\beta), \partial_\alpha\phi(\gamma_c(\alpha,\beta), \gamma_c(\alpha,\beta))) \in \Sigma.
\end{align*}
On the other hand, according to \eqref{eq:identities-delta-f}, $\delta(a,\delta(a,b))\in f(\delta(a,b))$. Since $\delta(a,b)$ and $\delta(a,\delta(a,b))$ share the same first component $\gamma_c(\alpha,\beta)$, they must actually be equal, and we deduce that $\delta(a,b)\in\Sigma$, which is what we wanted to prove. 

It remains to prove \eqref{eq:associativity}. We notice that the operation on (germs near the diagonal of) phase functions defined by
\[
\psi_1 \odot \psi_2(\alpha,\beta) = VC_\gamma (\psi_1(\alpha,\gamma)+\psi_2(\gamma,\beta)),
\]
is associative on the set of phases for which the critical value is uniquely defined via an implicit function argument. Indeed, $\psi_1 \odot ( \psi_2 \odot \psi_3)= (\psi_1 \odot \psi_2)\odot \psi_3$ at $(\alpha,\beta)$ is just
\[
VC_{\delta,\eta}\left[\psi_1(\alpha,\delta)+\psi_2(\delta,\eta)+\psi_3(\eta,\beta)\right]
\]
Then we write
\[
\phi(\alpha,\beta) = VC_{\delta,\eta}\left[ \phi(\alpha,\delta)+\phi(\delta,\eta)+\phi(\eta,\beta)\right],
\]
attained at $(\delta_c,\eta_c)$. We find
\[
\delta_c = \gamma_c(\alpha,\eta_c),\ \eta_c = \gamma_c(\delta_c, \beta).
\]
But according to \eqref{eq:d-critique}, we have $\partial_\alpha \phi(\delta_c,\eta_c) = \partial_\alpha\phi(\delta_c,\beta)$, so that
\begin{equation*}
\partial_\beta \phi(\alpha,\delta_c) = -\partial_\alpha \phi(\delta_c, \eta_c) = - \partial_\alpha\phi(\delta_c, \beta). 
\end{equation*} 
This implies $\delta_c = \gamma_c(\alpha,\beta)$, i.e \eqref{eq:associativity}.

\subsubsection{The involutive manifold \emph{is} a manifold}

We start by a lemma
\begin{lemma}
In the case of $\phi^\alpha$, the smoothness and transversal intersection parts of Proposition \ref{prop:involutive} hold.
\end{lemma}
\begin{proof}
To prove this, it suffices to work with $\psi^\alpha$. In that case, we find that
\begin{align}
\mathcal{J}_{\psi^\alpha} &= \left\{(u, \vartheta_\alpha L^{-1} + i u - (\tilde{J}+i)v\ \middle|\ u,v\right\}, \\
\mathcal{J}_{\psi^\alpha}^\ast &= \left\{(v, \vartheta_\alpha L^{-1} + i v - (\tilde{J}-i)u\ \middle|\ u,v\right\}, 
\end{align}
The spectral decomposition of $\tilde{J}$ then ensures the statement is true.
\end{proof}

Now, if we can prove that $\partial_\alpha\partial_\beta\phi$ has constant rank $n$, the proof of Proposition \ref{prop:involutive} will be complete. Indeed, this will prove that $\mathcal{J}_\phi$ and $\mathcal{J}_\phi^\ast$ are smooth $3n$-dimensional manifolds. The transversality condition will carry on from the transversality for the $\phi^\alpha$'s. Finally, the fact that $\mathcal{J}_\phi$ is involutive, i.e. that
\[
\omega(u,v)=0,\ \forall v\in T \mathcal{J}_\phi  \ \Longrightarrow u\in T \mathcal{J}_\phi,
\]
follows from the fact that $\mathcal{J}_\phi$ is a union of lagrangian manifolds (Proposition \ref{prop:union-lagrangian}).

Since the rank of $\partial_\alpha\partial_\beta\phi$ equals $n$ on the diagonal, it is at least equal to $n$ near the diagonal. Let us describe some properties of the map $\gamma_c$. According to its definition, letting $F(\alpha,\beta) =\partial_\alpha^2 \phi(\gamma_c,\beta)+\partial_\beta^2 \phi(\alpha,\gamma_c)$, (and writing $\gamma_c$ as a column vector),
\[
\partial_\alpha \gamma_c (\alpha,\beta)= -F^{-1} \partial_\beta\partial_\alpha\phi(\alpha,\gamma_c),\quad \partial_\beta\gamma_c = - F^{-1} \partial_\alpha\partial_\beta\phi(\gamma_c,\beta).
\]
On the diagonal, we get that
\[
-F \partial_\alpha\gamma_c = B = - R - i D = - D( J + i) ,\quad -F \partial_\beta \gamma_c = D(J - i). 
\]
In particular, both have rank $n$, and their images are transverse, so that $\partial_{\alpha,\beta}\gamma_c$ has maximal rank $2n$ on the diagonal (and thus also on a small neighbourhood of the diagonal).
\begin{lemma}
The rank of $\partial_\alpha\partial_\beta \phi$ is $n$ everywhere. 
\end{lemma}

\begin{proof}
For $\gamma$ in $U$, we define 
\[
\mathcal{F}^c(\gamma) = \{ (\alpha,\beta) \ |\ \gamma_c(\alpha,\beta) = \gamma \}. 
\]
This is certainly a well defined $2n$-submanifold of $U\times U$ near $(\gamma,\gamma)$. The family $\{ \mathcal{F}^c(\gamma)\}_{\gamma}$ defines a foliation near the diagonal. From the definition of $\gamma_c(\alpha,\beta)$, we get for $(\alpha,\beta) \in \mathcal{F}^c(\gamma)$,
\[
-\partial_\beta \phi(\alpha,\gamma) = \partial_\alpha \phi(\gamma, \beta). 
\]
However, 
\[
(\gamma, - \partial_\beta\phi(\alpha,\gamma)) \in \mathcal{J}_\phi^\ast,\quad (\gamma,\partial_\alpha \phi(\gamma,\beta)) \in \mathcal{J}_\phi, 
\]
so that we get from \eqref{eq:intersection of jphi}
\begin{equation}\label{eq:constant-dtheta}
\partial_\alpha \phi(\gamma, \beta) = - \partial_\beta \phi(\alpha,\gamma) = \vartheta_\gamma. 
\end{equation}
Denoting $\pi_1$ and $\pi_2$ the first and second projection in $\C^{2n}\times \C^{2n}$, we differentiate \eqref{eq:constant-dtheta} along $\mathcal{F}^c(\gamma)$ to find that
\[
u\in \pi_1 T_{(\alpha,\beta)}\mathcal{F}^c(\gamma) \Longrightarrow \partial_\beta \partial_\alpha \phi(\alpha,\gamma) u = 0. 
\]
In particular, the dimension of $\pi_1 T_{(\alpha,\beta)} \mathcal{F}^c(\gamma)$ cannot exceed $n$. Likewise replacing $\pi_1$ by $\pi_2$. However, 
\[
\begin{split}
\dim \pi_1 T_{(\alpha,\beta)} \mathcal{F}^c(\gamma) & + \dim\pi_2 T_{(\alpha,\beta)} \mathcal{F}^c(\gamma)  = \\
			& 2n  + \dim( \ran(\partial_\alpha \gamma_c) \cap \ran(\partial_\beta \gamma_c)).
\end{split}
\]
It follows that 
\begin{equation}\label{eq:TFgamma is good}
\begin{split}
\ran(\partial_\alpha \gamma_c) &\cap \ran(\partial_\beta \gamma_c) = \{0\},\\
		& n = \dim \pi_1 T_{(\alpha,\beta)} \mathcal{F}^c(\gamma) = \dim \pi_2 T_{(\alpha,\beta)} \mathcal{F}^c(\gamma).
\end{split}
\end{equation}
This shows that $\partial_\beta\partial_\alpha\phi(\alpha,\gamma)$ has rank $n$. However it is the rank of $\partial_\beta\partial_\alpha\phi(\alpha,\beta)$ we are after. From the projector property, recall that
\[
\partial_\alpha\phi(\alpha,\beta) = \partial_\alpha \phi(\alpha,\gamma_c(\alpha,\beta)). 
\]
Differentiating this equality only in $\beta$, we find
\[
\partial_\beta\partial_\alpha\phi(\alpha,\beta) = (\partial_\beta \gamma_c)^\top \partial_\beta\partial_\alpha\phi(\alpha,\gamma).
\]
In particular,
\[
v\in \ker \partial_\beta\partial_\alpha\phi(\alpha,\gamma)\  \Longrightarrow \ \partial_\beta\partial_\alpha\phi(\alpha,\beta)v = 0. 
\]
This shows that $\partial_\alpha\partial_\beta\phi(\alpha,\beta)$ has rank at most $n$, concluding the proof. 
\end{proof}

A consequence of the proof is the following:
\begin{equation}\label{eq:splitting}
T_{(\alpha,\beta)}\mathcal{F}^c(\gamma) = \{ (u,v) \ |\ \partial_\beta\partial_\alpha \phi(\alpha,\beta) u=  \partial_\alpha\partial_\beta\phi(\alpha,\beta) v = 0 \}. 
\end{equation}
Another key consequence is the following. The first equality in \eqref{eq:TFgamma is good} implies in particular that 
\[
\ran( \partial_\beta\partial_\alpha\phi ) \cap \ran( \partial_\alpha\partial_\beta\phi) = \{ 0 \},
\]
or equivalently
\begin{equation}\label{eq:tranversality}
\ker( \partial_\beta\partial_\alpha\phi ) \cap \ker( \partial_\alpha\partial_\beta\phi) = \{ 0 \},
\end{equation}

\subsubsection{Factorisation of the phase}

Since $\mathcal{J}_\phi$ is an involutive manifold, we can define its null foliation. 
\begin{definition}\label{def:null-foliation}
The kernel of $\omega_{|T\mathcal{J}_\phi}$ defines an involutive distribution, tangent to the so-called null foliation $\mathcal{F}_\phi$ of $\mathcal{J}_\phi$. It is given by 
\[
T_{(\alpha,\partial_\alpha\phi(\alpha,\beta))}\mathcal{F}_\phi = \{ (u, \partial_{\alpha}^2\phi u) \ |\ u\in \ker \partial_\beta\partial_\alpha\phi\}. 
\]
Additionally, $\Sigma_\phi$ is transverse to $\mathcal{F}_\phi$. 
\end{definition}

\begin{proof}
We have that ($u,v$ as column vectors)
\[
T_{(\alpha,\partial_\alpha\phi(\alpha,\beta))}\mathcal{J}_\phi = \{ (u, \partial_\alpha^2 \phi(\alpha,\beta)u )\ |\ u \in \C^{2n} \} \oplus \{ (0, \partial_\alpha\partial_\beta\phi v )\ |\ v\in\C^{2n}\}. 
\]
We take $w=(u,\partial^2_\alpha \phi u )+(0,\partial_\alpha\partial_\beta \phi v)$. Then (using the fact that $\partial_\alpha^2\phi$ is symmetric) $w\in\ker \omega_{|T\mathcal{J}_\phi}$ if and only if for all $u',v'$,
\[
\langle u, \partial_\alpha\partial_\beta\phi v'\rangle - \langle u', \partial_\alpha\partial_\beta\phi v \rangle =0,
\]
i.e
\[
\partial_\alpha\partial_\beta\phi  v = 0, \text{ and } \partial_\beta\partial_\alpha\phi u= 0. 
\]
So we get
\[
\ker \omega_{|T\mathcal{J}_\phi} = \{ (u, \partial^2_\alpha \phi u)\ |\ u\in \ker \partial_\beta\partial_\alpha \phi \}. 
\]
We check that $\Sigma_\phi$ is transverse to $\mathcal{F}_\phi$, by recalling that
\[
T\Sigma_\phi = \{ (u, \partial_\alpha^2 \phi u+ \partial_\alpha\partial_\beta\phi u) \ |\ u \},
\]
and using \eqref{eq:tranversality}. 
\end{proof}

We can obviously define a similar foliation $\mathcal{F}^\ast_\phi$ for $\mathcal{J}_\phi^\ast$. Let us now study the relation between them and $\mathcal{F}^c$ defined in the previous section. 

Define
\[
\Lambda_\phi = \{ (\alpha,\partial_\alpha \phi(\alpha,\beta),  \beta, -\partial_\beta\phi(\alpha,\beta)) \ |\ \alpha,\beta \}. 
\]
This is a lagrangian manifold in $T^\ast U\times U$, (i.e dimension $4n$) whose projection on the first  (resp. second) pair of variables is $\mathcal{J}_\phi$ (resp. $\mathcal{J}_\phi^\ast$). 

Recall from \S\ref{sec:associativity} the sets $f(b)$. Since the rank of $\partial_\alpha\partial_\beta\phi$ is constant, these are smooth sets. Observe that when $b= (b,\eta)$, and $a=(\alpha,\xi)\in f(b)$, 
\[
T_a f(b) = \{ (u, \partial^2_\alpha \phi(\alpha,\beta)) \ |\ \partial_\beta\partial_\alpha \phi(\alpha,\beta) u = 0 \}.
\]
This means that $f(b)$ is actually tangent at every point to the null foliation of $\mathcal{J}_\phi$. In particular, $f(b)$ is a \emph{leaf} of $\mathcal{F}_\phi$. Additionally, $\delta(a,b) \in \Sigma \cap f(b)$ is the unique point of intersection of this leaf with $\Sigma$. Since we also have $a\in f(\delta(a,b))$, we get $f(\delta(a,b)) = f(b)$, so that since
\[
\Lambda_\phi = \{ (a,b) \ |\ a\in f(b)\}, 
\]
we get
\begin{equation}\label{eq:pre-factorization of phi}
\Lambda_\phi = \cup_{\gamma\in\Sigma} \mathcal{F}_\phi(\gamma)\times \mathcal{F}^\ast_\phi(\gamma). 
\end{equation}

This suggests to consider
\[
\Lambda_{\to} = \{ (a, b) \ |\ a\in f(b),\ b\in \Sigma \}. 
\]
This is a submanifold of $T^\ast U \times \Sigma$, of dimension $3n$, so it has the right dimension to be a lagrangian, and indeed,

\begin{lemma}
The lagrangian $\Lambda_\to$ is strictly positive, with real points
\[
\Sigma_\R \times\Sigma_\R. 
\]
\end{lemma}

\begin{proof}
That the real points are given by $\Sigma_\R \times\Sigma_\R$ follows from the construction, and the fact that $\Sigma_\R$ is the set of real points of $\mathcal{J}_\phi$. We study now $\omega(v,\overline{v})$ for $v\in T\Lambda_\to$ at real points. 

Recall that the null-foliation is tangent to $\ker \omega_{|\mathcal{J}_\phi}$, which is given by
\[
T\mathcal{F}_\phi =\{(u, \partial_\alpha^2\phi(\alpha,\beta)u)\ |\ \partial_\beta\partial_\alpha \phi(\alpha,\beta)u=0\}. 
\]
To prove that $\Lambda_\to$ is positive, it suffices to prove that $i \omega(v,\overline{v}) >0$ on $(T\mathcal{F}_\phi)_{|\Sigma}$. This is
\[
i \langle u, \overline{\partial_\alpha^2\phi u}\rangle - i \langle \overline{u}, \partial_\alpha^2\phi u\rangle = 2\langle \Im(\partial_\alpha^2\phi)u, \overline{u}\rangle.
\]
By assumption, $\Im(\partial_\alpha^2\phi)>0$. 
\end{proof}

We can identify $\Sigma$ with $U$ with the map $\alpha\mapsto (\alpha,\vartheta_\alpha)$, and observe that the pullback of the symplectic form on $\Sigma$ is exactly $\omega$ on $U$. Choosing some Darboux coordinates in $U$ so that $\vartheta = \xi dx$ defines a local vertical direction on $\Sigma$, and thus on $T^\ast U \times \Sigma$.
\begin{lemma}
The lagrangian $\Lambda_\to$ is transverse to such a vertical direction. 
\end{lemma}

\begin{proof}
Choose Darboux coordinates $(x,\xi)$ in $U$, so that $\vartheta = \xi dx$. Identifying $U$ and $\Sigma$, we have
\[
\Lambda_\to = \{ (\beta_x, \beta_\xi,\beta_x^\ast,\beta_\xi^\ast;\alpha_x,\alpha_\xi)\ |\ (\beta,\beta^\ast)\in\mathcal{F}(\alpha_x,\alpha_\xi,\alpha_\xi,0)\}.
\]
Thus, at $\Sigma$, 
\[
T\Lambda_\to = \{ (u_x, u_\xi , u_\xi, 0 ; u_x, u_\xi)\ |\ u_x,u_\xi \} + \{( v , \partial_\alpha^2\phi v ; 0)\ |\ \partial_\beta\partial_\alpha\phi v=0\}. 
\]
The tangent to the vertical foliation $V$ is given here by 
\[
T_{(\alpha,\alpha^\ast, \beta)} V = \{ (0, v^\ast ; 0, w_\xi) \ |\ v^\ast, w_\xi\}.  
\]
A vector in both writes
\[
(0, \partial^2_\alpha\phi(0,-u_\xi) + (u_\xi,0); 0, u_\xi)
\]
with $(0,u_\xi)\partial_\alpha\partial_\beta \phi = 0$. However, here, we have
\[
R = \begin{pmatrix}
0&-1\\1&0
\end{pmatrix}
\]
so that
\[
(0,u_\xi)\partial_\alpha\partial_\beta \phi \begin{pmatrix} 0\\ \overline{u_\xi}\end{pmatrix} = - i (0,u_\xi)D \begin{pmatrix} 0\\ \overline{u_\xi}\end{pmatrix}=0
\] 
implies $u_\xi=0$, because the real part of $D$ is strictly positive and we are done. 

\end{proof}

\subsection{Local FBI transform}

From now on, we will be manipulating the tools of analytic microlocal analysis. We will assume that the reader is familiar with the main techniques from that field; essentially (non-)stationary phase expansions. Let $\Omega\subset \R^k$ be an open set. We will say that a function $a: \Omega\to \C$ is a real analytic symbol in $\Omega$ if $a$ is a family of real analytic functions $a_h = a(\cdot,h)$ on the set $\Omega$, depending on a real parameter $0<h<h_0$, satisfying an expansion\footnote{most of the time, the exponents $k$ will be the integers, but we also allow to sum over $k\in \{k_0 + \ell,\ \ell\in\N\}$, for some $k_0\in\R$.}
\[
a(x,h) = \sum_{k=k_0}^{1/h} h^k a_k(x)  + \mathcal{O}(e^{-c/h}),
\]
where $c$ is some constant, and 
\[
\mathbf{a} = \sum h^k a_k
\]
is its associated \emph{formal symbol}, that must satisfy estimates of the following form. There exists a complex open neighbourhood $\tilde{\Omega}\supset \overline{\Omega}$ where all the $a_k$'s have holomorphic extension, and satisfy
\[
|a_k| \leq C^k k !
\]
Then we say that $a$ is a realization of $\mathbf{a}$; a formal symbol may have many different realizations. Note that here, we do not allow the $a_k$'s to depend on $h$. Sometimes, we will drop the word \emph{analytic}, but implicitly we are working in the analytic case. All the arguments in this section can be adapted to the $C^\infty$ case using almost analytic extensions as in \cite{Melin-Sjostrand-75}, we will not go into these details. 

Let us get back to our problem; We will work locally around some point $\alpha_0\in U$, and will assume to have chosen analytic Darboux coordinates around $\alpha_0$, so that it makes sense to consider $A\times B\subset U$ an open ball neighbourhood of $\alpha_0$. Without loss of generality, we can assume that $\alpha_0=0$. We can pick $A$ and $B$ small enough so that all the constructions that before were valid only near the diagonal, actually hold without this restriction. Also pick $E$ a ball around $0$ in $\R^n$. We can apply the results of \S \ref{sec:study-phase} to $\tilde{U}=A\times B$. 

Since $\Lambda_\to$ is projectable and positive, we can find a phase $\psi$ defined in a neighbourhood of $(0,0,0)$, so that
\begin{enumerate}
	\item $\Im \psi \geq 0$
	\item $\Im \psi(\alpha,x)=0$ if and only if $\alpha_x = x$. 
	\item Locally,
	\[
	\Lambda_\to = \{ (\alpha,\partial_\alpha\psi, x, -\partial_x\psi)\ |\ \alpha,x\}. 
	\]
\end{enumerate}

\begin{proof}
We just have to check (2). For this we take a function $\psi_0$ satisfying (3) and observe that along $(\Lambda_\to)_\R$, we have 
\[
d\Im \psi = 0,\quad d^2 \Im \psi >0.
\] 
(because $\Lambda_\to$ is positive). Thus $\psi = \psi_0  - i (\Im \psi_0)_{|(\Lambda_\to)_\R}$ is suitable.
\end{proof}

We have 
\begin{equation}\label{eq:Taylor-psi}
\psi(\alpha,x) = \langle \alpha_x - x, \alpha_\xi\rangle + \mathcal{O}(|\alpha_x-x|^2). 
\end{equation}
The same construction can be made for $\Lambda_\leftarrow$, giving a phase $\psi^\ast$, with
\begin{equation}
\psi^\ast(x,\beta) = \langle x-\beta_x,\beta_\xi\rangle + \mathcal{O}(| \beta_x-x|^2). 
\end{equation}

Define then for $a$ an analytic symbol in $A\times B \times A$,
\[
T_a u(\alpha) = \frac{1}{2^{n/2} (h\pi)^{3n/4}}\int_A e^{\frac{i}{h}\psi (\alpha,x)} a_h(\alpha,x) u(x) dx,
\]
and
\[
S_a u(x) =  \frac{1}{2^{n/2} (h\pi)^{3n/4}} \int_{A\times B} e^{\frac{i}{h}\psi^\ast(x,\beta)}a_h(\beta,x) u(\beta) d\beta. 
\]
We will call $T_a$ a local FBI transform and $S_a$ a local adjoint FBI transform. In \cite{Singularite-analytique-microlocale}, Sj\"ostrand defined local FBI transforms with the assumption that there is a complex structure on the side of the variable $\alpha$, and requiring that $T_a$ maps into holomorphic functions for this structure. For this reason the FBI transforms in \cite{Singularite-analytique-microlocale} are operators that map functions of $n$ variables to functions of $n$ variables. Our situation is slightly different, because we only have an \emph{almost} complex structure on the $\alpha$ side, and thus our transforms map into functions of $2n$ variables. For $p$ an analytic symbol in $A\times B$, we also define
\[
\Op(p)u(x) = \frac{1}{(2\pi h)^n}\int_{A\times B} e^{\frac{i}{h}\langle x-y,\xi\rangle} p(x,\xi) u(y) dy d\xi. 
\]

Since we do not want to have problems with the boundary, we have to make our operators act on spaces of functions that are very small near the boundary. To avoid working with the $H_\varphi$ spaces of \cite{Singularite-analytique-microlocale}, we introduce a space of localized functions, that behave like gaussian wave packets. In the proof of Lemma \ref{lemma:inverse-FIO}, we will use $H_\varphi$ spaces, but almost completely as a black box; for the rest, we will not need them.
\begin{definition}
Let $\mathcal{C}(A,B)$ be the set of $h$-dependent measurable functions $f$ on $A$ taking the form
\[
f(x)= e^{\frac{i}{h} S(x)} \sigma_h(x) + \mathcal{O}(e^{-1/Ch}),
\]
with $S$ real analytic, $\sigma_h$ a real analytic symbol. $S$ is real at only one point $x(S)$ of $A$, and satisfies
\begin{equation}\label{eq:positivity-C(A,B)}
\Im S \geq |x-x(S)|^2/C,
\end{equation}
and
\begin{equation}\label{eq:oscillating-right-direction}
dS(x(S)) \in B. 
\end{equation}
We also allow $S$ and $\sigma_h$ to be only defined in an open set of $A$, provided $\Im S > 1/C$ on the boundary of this set, and they satisfy uniform estimates in this set. In that case, outside of the set where $S$ and $\sigma_h$ are defined, $f$ is just assumed to be exponentially small.

On the $\alpha$ side, we define $\mathcal{C}(A\times B)$ to be functions in $\mathcal{C}(A\times B, B\times E)$ (recall $E$ is a small ball around $0$), satisfying the additional condition:
\begin{equation}\label{eq:condition-real-point-C(AxB)}
(\alpha_x(S), \alpha_\xi(S), \partial_{\alpha_x} S, \partial_{\alpha_\xi} S) \in \Sigma_{\R}.
\end{equation}
I.e, at the real point,
\[
\partial_{\alpha_x} S=\alpha_\xi,\quad \partial_{\alpha_\xi} S = 0.
\]
We can put a topology on $\mathcal{C}(A,B)$, specifying that bounded sets in $\mathcal{C}(A,B)$ are defined using semi norms on the relevant objects, and taking $x(S)$ and $dS(x(S))$ to live in compact sets of $A$, $B$. 
\end{definition}

When we say that acting on $\mathcal{C}(A,B)$, some operator $P$ is $\mathcal{O}(e^{-1/Ch})$, we mean that for any $f\in\mathcal{C}(A,B)$, there exists $C>0$ depending continuously on $f$, such that $Pf = \mathcal{O}(e^{-1/Ch})$.

\begin{proposition}
The following maps are continuous
\begin{align*}
\Op(p)&: \mathcal{C}(A,B) \to \mathcal{C}(A,B),\\
T_a &: \mathcal{C}(A,B) \to \mathcal{C}(A\times B),\\
S_b &: \mathcal{C}(A\times B) \to \mathcal{C}(A,B). 
\end{align*}
\end{proposition}

\begin{proof}
The proof of all three statements is the same. Writing the expression for the operator, one realizes it is the sum of an exponentially small term and a stationary phase integral. For the exponentially small remainder, it suffices to use $L^\infty$ bounds to obtain continuity. For the stationary integral contribution, we see that the phase satisfies
\begin{enumerate}
	\item non-negative imaginary part
	\item uniformly positive imaginary part on the boundary of the domain on integration
	\item There is exactly one parameter for which the phase is real at the stationary point.
\end{enumerate}
We can then directly use the holomorphic stationary phase theorem \cite[Théorème 2.8 and Remarque 2.10]{Singularite-analytique-microlocale}. This proves that the maps are well defined, provided we check condition \eqref{eq:oscillating-right-direction}. In the case of $\Op(p)$, we observe that $\Op(p)$ does not even change the phase and only acts on the amplitude, so \eqref{eq:oscillating-right-direction} is certainly satisfied. Now we deal with $T_a$. Here, if the phase associated with $f$ is $S$, with real point $(x(S), dS(x(S)))$, then $T_a f$ is associated with the phase $\tilde{S}$, given as
\begin{equation}\label{eq:defining-Stilde}
\tilde{S}(\alpha) = VC_{x} (\psi(\alpha,x) + S(x)). 
\end{equation}
We notice that according to \eqref{eq:Taylor-psi}
\[
\partial_x\left( \psi(x(S), dS(x(S)), x) + S(x)\right)_{|x=x(S)} = 0,
\]
so that
\[
\tilde{S}(x(S),dS(x(S))) = S(x(S)), 
\]
and at that point,
\[
\partial_{\alpha_x} \tilde{S} = dS(x(S)),\text{ and } \partial_{\alpha_\xi} \tilde{S} = 0. 
\]
This means that condition \eqref{eq:condition-real-point-C(AxB)} is satisfied. Now, we claim that condition \eqref{eq:positivity-C(A,B)} means that the graph $\Lambda$ of $d\tilde{S}$ must be a positive lagrangian with single real point. Since it is contained in $\mathcal{J}_\phi$, we can use Lemma \ref{lemma:positive-lagrangian-inside-J} and see that it suffices to consider the intersection of $\Lambda$ with $\Sigma$. That is, points where
\[
\partial_{\alpha_x} \tilde{S} = \alpha_\xi,\quad \partial_{\alpha_\xi} \tilde{S} = 0. 
\]
However, from \eqref{eq:Taylor-psi}, we see that this implies the critical point $x$ in \eqref{eq:defining-Stilde} is equal to $\alpha_x$, and $\alpha_\xi= \partial_x S$. The positivity of $S$ implies then the positivity of $\Lambda\cap \Sigma$.

A similar computation can be carried out for $S_b$. 

To prove the maps are continuous it suffices to invoke the fact that the constants that appear in stationary phase estimates can always be controlled in terms of symbol norms, and some lower bounds on the hessian, and the imaginary part of the phase near the boundary. 
\end{proof}

Notice that if we consider $S_b f$, with $f \in \mathcal{C}(A\times B, B\times E)$ whose ``real point'' is not contained in $\Sigma_\R$, the result will be $\mathcal{O}(e^{-c/h})$, because $S_b f$ will be a complex Lagrangian state, whose Lagrangian has no real point. 

Let us discuss the action of pseudors. For $p,q$ two symbols, we denote $p\# q$ a fixed realization of the formal symbol
\[
\sum_{\nu\in \N^n} \frac{(ih)^\nu}{\nu !} \partial_\xi^\nu p\ \partial_x^\nu q. 
\]
\begin{proposition}
For $p,q$ analytic symbols in $A\times B$, given $f\in \mathcal{C}(A,B)$, 
\[
\Op(p)\Op(q) f = \Op(p\# q) f + \mathcal{O}(e^{-1/Ch}). 
\]
\end{proposition}

\begin{proof}
Let us write out the integral considered here in the form
\[
\frac{1}{(2\pi h)^{2n}}\int_{(A\times B)^2} e^{\frac{i}{h}( \langle x-z, \eta\rangle + \langle z-y,\xi\rangle + S(y))}p(x,\eta)q(z,\xi) \sigma(y) dz d\xi dy d\eta. 
\]
Up to the normalization factor, this is
\[
\int_{A\times B^2}e^{\frac{i}{h}( \langle x,\xi\rangle + \langle z-x, \xi - \eta \rangle)} p(x,\eta)q(z,\xi) \left[\int_A e^{\frac{i}{h}(- \langle y, \xi\rangle + S(y))}\sigma(y) dy\right] dz d\xi d\eta. 
\]
The integral in brackets is, according to stationary phase, and up to some universal function of $h$, of the form
\[
e^{\frac{i}{h}\tilde{S}(\xi)}\tilde{\sigma}(\xi)+\mathcal{O}(e^{-1/Ch}). 
\]
Here $\tilde{S}$ is the Legendre transform of $S$, and its imaginary part is uniformly positive near $\partial B$.

According to stationary phase, 
\[
\begin{split}
\int_{A\times B} & e^{\frac{i}{h}\langle z-x,\xi- \eta\rangle} p(x,\eta)q(z,\xi) dz d\eta \\
		&\hspace{-8pt}= (2\pi h)^n p\#q(x,\xi) + \mathcal{O}(e^{-1/Ch}) + \int_{\Gamma} e^{\frac{i}{h}\langle z-x,\xi- \eta\rangle}p(x,\eta)q(z,\xi) dz d\eta.
\end{split}
\]
Here, we tried to replace integration over $A\times B$ by integration over a contour for which the imaginary part of $\langle z-x, \xi-\eta\rangle$ is positive near the boundary, to be able to apply a stationary phase estimate estimate. The price to pay was the appearance of the boundary term $\int_\Gamma$, where $\Gamma$ is a contour depending on $x$ and $\xi$, so that for $(z,\eta)\in \Gamma$, $\Re(z,\eta)\in \partial(A\times B)$, and $\Im(\langle z-x,\xi-\eta \rangle)\geq 0$ on $\Gamma$. The exponential error is here because we may have had to change the realization of the formal symbol. 

It remains to consider
\[
\int_B e^{\frac{i}{h}(\tilde{S}(\xi)+ \langle x,\xi\rangle)} \int_{\Gamma} e^{\frac{i}{h}\langle z -x, \xi-\eta\rangle} p(x,\eta)q(z,\xi) dzd\eta d\xi. 
\]
When $\xi$ is near $\partial B$, the integrand is exponentially small. Additionally, the phase is only real when $\xi = d_{x(S)}S$, which is inside $B$. At such a point, the phase is not stationary in $\xi$, since its derivative with respect to $\xi$ is given by 
\[
d_\xi \tilde{S} + z = z - x(S) \neq 0. 
\] 
We can thus apply a non-stationary phase contour deformation argument to prove that this contribution is exponentially small. 
\end{proof}

\begin{proposition}\label{prop:pseudos-downstairs}
The compositions $T_a S_b$ and $S_b T_a$ are continuous respectively on $\mathcal{C}(A\times B)$ and $\mathcal{C}(A,B)$. Additionally, we can find $c$ an analytic symbol on $A\times B$ such that acting on $\mathcal{C}(A,B)$,
\[
S_b T_a f = \Op(c) f + \mathcal{O}(e^{-1/Ch}).
\]
with $c$ depending continuously on $a,b$ as an analytic symbol, and
\[
c = ab \sigma + \mathcal{O}(h),
\]
where $\sigma$ is a real analytic function, depending only on $\psi$ and $\psi^\ast$\footnote{We have the formula $\sigma(x,\alpha_\xi) = \left(\det \frac{1}{i} \partial^2_{\alpha_x}(\psi^\ast(x,\alpha)+\psi(\alpha,y))_{|x=y=\alpha_x} \right)^{-1/2}$.}
\end{proposition}

The proof of this lemma follows the ideas of the proof of Lemma 2.8 in \cite{GBJ}, but the general idea is well known to specialists.
\begin{proof}
Let us start by observing that
\[
f(y) = \frac{1}{(2\pi h)^n} \int_{A\times B} e^{\frac{i}{h}\langle y-z,\eta\rangle} f(z) dz d\eta + \mathcal{O}(e^{-1/Ch}),
\]
with the exponentially small remainder depending continuously on $f$. We insert this identity into $S_b T_a f(x)$. Up to a normalization factor $w(h)$ and an exponentially small error, we obtain for $S_b T_a f(x)$
\[
\int_{(A\times B) \times A \times (A \times B)} \hspace{-30pt} e^{\frac{i}{h}\left( \psi^\ast(x,\alpha)+\psi(\alpha, y) + \langle y-z, \eta\rangle\right)} b(x,\alpha)a(\alpha,y) f(z) d\alpha dy dz d\eta. 
\]
We can rewrite this slightly in the form
\begin{equation}\label{eq:nice-expression-for-S_bT_af}
\int_{A\times B} \hspace{-15pt} e^{\frac{i}{h}\langle x-z,\eta\rangle} f(z) \left(\int_{A\times (A\times B)} \hspace{-30pt} e^{\frac{i}{h}\left(\psi^\ast(x,\alpha)+\psi(\alpha,y)+\langle y-x,\eta\rangle \right)}b(x,\alpha) a(\alpha,y) dy d\alpha \right)dz d\eta
\end{equation}
It suffices thus to study 
\[
\tilde{c}(x,\eta)= w'(h)\int_{A\times (A\times B)} \hspace{-30pt} e^{\frac{i}{h}\left(\psi^\ast(x,\alpha)+\psi(\alpha,y)+\langle y-x,\eta\rangle \right)}b(x,\alpha) a(\alpha,y) dy d\alpha
\]
($w'(h)$ being the relevant normalization factor). The phase function $\Psi$ here satisfies
\[
\Psi = \langle x-y,\alpha_\xi-\eta\rangle + \mathcal{O}(|\alpha_x-x|^2+|\alpha_x-y|^2),
\]
and
\[
\Im \Psi \geq \frac{1}{C}( |x-y|^2 + |\alpha_x - (x+y)/2|^2). 
\]
As long as $x$ does not encounter the boundary $\partial A$, we have a uniform positive lower bound on the imaginary part of $\Psi$ near the boundary $(y,\alpha_x)\in \partial (A^2)$. When $\alpha_\xi$ is close to the boundary $\partial B$, we do \emph{not} have a positive lower bound on $\Im \Psi$. However, we observe that whenever $\Im \Psi$ is small,
\[
|\partial_{y,\alpha_x} \psi | \geq \frac{1}{C} |\alpha_\xi-\eta|. 
\]
In particular, provided $\eta$ is not too close to $\partial B$, we can apply a non-stationary phase argument in $(y,\alpha_x)$ to deform the boundary of the contour of integration, to ensure that the imaginary part of $\Psi$ satisfies a global positive lower bound on the boundary of the contour. 

We deduce that stationary phase applies to $\tilde{c}$, showing that for any $A'\times B'$ relatively compact in $A\times B$, $\tilde{c}$ is an analytic symbol on $A'\times B'$. The next step is to notice that if $\mathbf{c}$ is the formal symbol associated with $\tilde{c}$, it is obtained by a stationary phase expansion, and thus its coefficients can be expressed in $A\times B$ in terms of some universal coefficients, and the derivatives of $\psi^\ast$, $\psi$, $a$ and $b$, at points which remain inside the domains where they satisfy uniform symbol estimates. We deduce that $\mathbf{c}$ is a formal analytic symbol on $A\times B$, and we can thus pick a realization $c$ that is an analytic symbol on $A\times B$. 

Now, if we replace $\tilde{c}$ by $c$ in \eqref{eq:nice-expression-for-S_bT_af}, we will be making an error, that is exponentially small \emph{locally} uniformly in $A\times B$. However, since $f\in\mathcal{C}(A,B)$, we know that $f$ is already itself exponentially small near the boundary, so that we can absorb this loss of uniformity, uniformly in bounded sets of $\mathcal{C}(A,B)$.
\end{proof}

We close this section with two observations.  Operators of the form $T_a S_b$ have the form we were seeking. 
\begin{proposition}
Let $a$, $b$ be analytic symbols. Then acting on $\mathcal{C}(A\times B)$, 
\[
T_a S_b f (\alpha)= \frac{1}{(2\pi h)^n}\int_{A\times B} e^{\frac{i}{h}\phi(\alpha,\beta)} c(\alpha,\beta)f(\beta)d\beta + \mathcal{O}(e^{-1/Ch}),
\]
where $c$ is an analytic symbol defined in a neighbourhood of the diagonal, and $\epsilon>0$ is small enough. 
\end{proposition}

\begin{proof}
By construction of $\psi$ and $\psi^\ast$, it suffices to prove that we can apply a stationary phase argument. Let us consider the integral defining $T_a S_b f$, up to the normalization factor:
\[
\int_{A\times A\times B} e^{\frac{i}{h}(\psi(\alpha,x)+\psi^\ast(x,\beta))} a(\alpha,x) b(x,\beta) f(\beta)dxd\beta. 
\]
When $\beta$ is near the boundary of $A\times B$, the integrand is exponentially small because $f\in \mathcal{C}(A\times B)$. When $\beta$ is not close to the boundary, but $x$ is, then $\Im \psi^\ast(x,\beta)$ is positive, so that all in all, the contributions to the integral from the boundary are exponentially small, and this is all we need to apply stationary phase in $x$. 

Also as before, changing realization for $c$, we can take $c$ independent of $f$, up to adding an exponential error. 
\end{proof}

We decided to study the action of our operators on a very particular set of functions, and we need at some point to come back to the action on all smooth functions, i.e consider the actual Schwartz kernel. 
\begin{proposition}\label{prop:FromC(AxB)toKernel}
Let $A$ be an operator with kernel of the form $e^{\frac{i}{h}\phi} b+ \mathcal{O}(e^{-1/Ch})$, with $b$ an analytic symbol. Assume that on $\mathcal{C}(A\times B)$, $A$ is $\mathcal{O}(e^{-1/Ch})$. Then $b=\mathcal{O}(e^{-1/Ch})$. 
\end{proposition}

\begin{proof}
Notice that for $\gamma\in A\times B$,
\[
f_\gamma(\beta)=  e^{\frac{i}{h} \phi(\beta,\gamma)}
\]
defines an element of $\mathcal{C}(A\times B)$. By stationary phase, we have an expansion for $A f_\gamma(\alpha)$ in the form $e^{\frac{i}{h}\phi(\alpha,\gamma)}c(\alpha,\gamma) + \mathcal{O}(e^{-1/Ch})$. This ensures that near the diagonal $c=\mathcal{O}(e^{-1/Ch})$. However, studying the coefficients of the expansion given by stationary phase, this implies that the formal symbol of $b$ vanishes in a neighbourhood of the diagonal, and thus vanishes everywhere.  
\end{proof}

\subsection{Construction of the projector}
\label{sec:projector}

Our purpose now is to construct operators that satisfy \eqref{eq:Def-approximate-projector}, i.e approximate projectors. We will drop ``approximate'' to lighten the redaction. This is not completely ludicrous because, we could also decide to consider their action modulo exponentially small remainders, with $H_\varphi$-like spaces. 

In this section, we do not assume anymore that $U$ is simply connected. We start with the standard argument
\begin{proposition}
We can find symbols $a,b$ such that acting on $\mathcal{C}(A,B)$,
\[
S_b T_a = 1 + \mathcal{O}(e^{-1/Ch}).
\]
As a consequence, on $\mathcal{C}(A\times B)$, 
\[
(T_a S_b)^2 = T_a S_b + \mathcal{O}(e^{-1/Ch}). 
\]
If $\phi$ is a self-adjoint phase, we can ensure that $T_a S_b$ is formally self-adjoint. 
\end{proposition}

\begin{proof}
We observe that
\[
S_1 T_1 = \Op(c) + \mathcal{O}(e^{-1/Ch}),
\]
where $c=\sigma+\mathcal{O}(h)$ is an elliptic symbol, so that we can invert it (modulo exponential errors) in the form $\Op(d)$, and let 
\[
S_b = S_1,\quad T_a = T_1 \Op(d). 
\]
In the case that the phase is self-adjoint, there is a better thing to do. Indeed, then $S_1 = T_1^\ast$, and $\Op(c)$ is self adjoint and positive, so that we can find $d$ with $\Op(d)$ formally positive self-adjoint, and $\Op(c) = \Op(d)^2 + \mathcal{O}(e^{-1/Ch})$. Then we let 
\[
T_a = T_1 \Op(d),\quad S_b = (T_a)^\ast = \Op(d) S_1. 
\]
\end{proof}

Now, this construction provides a local projector, but it is not quite clear how one would patch up local projectors to get a global one. We will now improve this argument and see that there are many global (on $U$) projectors. Let us define a class of functions on which we can act without problems: covering $U$ by a union of $A\times B$ (each with its own Darboux coordinates) we let $\mathcal{C}(U)$ be the set of functions on $U$ that are exponentially small everywhere (locally uniformly), except in some $A\times B$, where they belong to $\mathcal{C}(A\times B)$ as defined before. Now
\begin{proposition}\label{prop:construction-many-projectors}
Let $p(\alpha,\beta)$ be an analytic symbol defined in a neighbourhood of the diagonal of $U$ (continued by zero elsewhere), non-vanishing on the diagonal (i.e elliptic). Then let $P$ be the operator with kernel $e^{i\phi(\alpha,\beta)/h} p(\alpha,\beta)$. There exist elliptic analytic symbols $b(\alpha,\beta)$ and $\mathfrak{p}(\alpha)$ such that the operator $\Pi$ with kernel $e^{i\phi/h}b$ satisfies (when acting on $\mathcal{C}(U)$ or on $L^2(U)$)
\begin{align*}
P &= \Pi \mathfrak{p} \Pi + \mathcal{O}(e^{-1/Ch}). \\
				\Pi ^2 & = \Pi + \mathcal{O}(e^{-1/Ch}). 
\end{align*}
\end{proposition}

\begin{proof}
We start by revisiting the local argument presented above. In a sufficiently small open patch $A\times B\subset U$ as in the previous section, we can find local FBI transforms $T$, $S$. Let
\[
T_f = P T_1,\quad S_g = S_1 P. 
\]
Then
\[
S_g T_f = Q
\]
is an elliptic pseudo, that we can invert, taking $T_m = T_f Q^{-1}$. Then we find (on $\mathcal{C}(A\times B)$)
\begin{equation}\label{eq:local-projector-C(AxB)}
(T_m S_g)^2 = T_m S_g T_m S_g = T_m Q Q^{-1} S_g = T_m S_g + \mathcal{O}(e^{-1/Ch}). 
\end{equation}
The operator $T_m S_g$ has kernel of the form $e^{i\phi/h}b_0$, with $b_0$ an analytic elliptic symbol defined on $(A\times B)^2$. However, we can use Proposition \ref{prop:FromC(AxB)toKernel} to use the a priori regularity and form of the kernels of $T_m S_g$ and $(T_m S_g)^2$ to prove that \eqref{eq:local-projector-C(AxB)} holds on $L^2(A\times B)$. 

Next, we recall the classic construction dating back to \cite{Boutet-Guillemin} (see also \cite{Sjostrand-96-convex-obstacle}). At least locally, since $\mathcal{J}_\phi$ is an involutive manifold, we can find (see Corollary \ref{cor:local-annihilation}) analytic functions $\zeta_1^0,\dots, \zeta_n^0$ whose common zero set is exactly $\mathcal{J}_\phi$, whose differentials are independent on $\mathcal{J}_\phi$, and with
\[
\{\zeta_i^0,\zeta_j^0\} = 0. 
\]
With a BKW analysis, one can then construct full semi-classical symbols\footnote{The construction provides them in the form $\zeta_j = \zeta_j^0 + h \zeta_j^1 +\dots$, where each term is analytic (at least near $\Sigma_\R$), but it is not clear that the growth of $\zeta_j^k$ as $k$ becomes large is that of a formal analytic symbol. A good part of the problem is that there are many choices to be made, since what is really well defined is the ideal generated by the $\zeta_j$'s.} $\zeta_1,\dots, \zeta_n$ so that
\[
[\Op(\zeta_i), \Op(\zeta_j)] =  \mathcal{O}(h^\infty), 
\]
and
\[
\Op(\zeta_j) P = \mathcal{O}(h^\infty). 
\]
A similar construction of $\zeta_j^\ast$ can be done, so that
\[
P \Op(\zeta_j^\ast) = \mathcal{O}(h^\infty).
\]
Next, one observes that if an operator $C$ whose kernel is of the form $e^{i\psi/h}c$, with $c = c_0 + h c_1 + \dots$, satisfies $\Op(\zeta_j) C,C\Op(\zeta_j^\ast) = \mathcal{O}(h^\infty)$ for $j=1\dots n$, then the jets of the $c_k$'s at the diagonal satisfy transport equations transverse to the diagonal, so that the jet of $c$ is determined by its value on the diagonal.

Then, if one also requires that $C^2 = C + \mathcal{O}(h^\infty)$, this gives additional relations on the $c_k$'s on the diagonal, of the form
\[
c_0 = 1,\quad c_k = F(c_0,\dots, c_{k-1}),\ k\geq 1.
\]
In particular, if we see this set of equations as a set of equations on the total jet of $c$ at the diagonal, there exists exactly one solution to this equation. One deduces that there exists near the diagonal at most one formal \emph{analytic} symbol $\sum h^k c_k$ so that for a realization $c$, 
\[
\Op(\zeta_j) C = \mathcal{O}(h^\infty),\quad C \Op(\zeta_j^\ast) = \mathcal{O}(h^\infty),\quad C^2= C+\mathcal{O}(h^\infty). 
\]
By construction, the arguments above can be localized, so that existence and uniqueness is also local. The fact that $T_m S_g$ satisfy these equations locally imply that the unique solution in the space of jets is the jet of a formal analytic symbol, that we denote by $\mathbf{b}$. The amplitude $b_0$ of $T_m S_g$ is a local realization of $\mathbf{b}$. We can then take instead a global realization $b$, and this provides us with the announced projector $\Pi$. 

To finish the proof, we have to prove that $P$ writes as a Tœplitz. For this, the method originated in \cite{Boutet-Guillemin} is to use the $\Op(\zeta_j)$ from before. Taking a symbol $\mathfrak{s}$, one forms the operator $\Pi \mathfrak{s} \Pi$, with kernel $e^{i\phi/h}\sigma$ and observe that one can read the values of $\mathfrak{s}$ in the restriction of $\sigma$ to the diagonal, using a stationary phase expansion. From this one can build a formal symbol $\mathbf{p}$ such that for any $C^\infty$ realization $\mathfrak{p}_0$ of $\mathbf{p}$, one has
\[
P - \Pi \mathfrak{p}_0 \Pi = \mathcal{O}(h^\infty). 
\]
This is presented in details in \cite[\S 2.3.2]{GBJ}. The strength of the method is that actually $\mathbf{p}$ is uniquely determined\footnote{Here we use crucially that all the symbols are classical, i.e have expansions in powers of $h$, with coefficients not depending on $h$.}. However, it does not provide directly \emph{analytic} symbol estimates. To obtain such estimates, one strategy is to study more carefully the transport equations as in \cite{Deleporte-18}, \cite{Charles-19}, and observe that one can indeed recover analytic regularity. We will take another perspective altogether. We will prove that around any point of $U$, one can find an analytic symbol solving the problem. By local uniqueness, this will ensure that $\mathbf{p}$ is actually an analytic symbol on $U$. 

Let $\alpha_0\in U$, and $A\times B\subset U$ a small neighbourhood of $\alpha_0$. Observe that for $\mathfrak{a}$ an analytic symbol, with local FBI transforms $T_1$, $S_1$,
\[
S_1 \Pi \mathfrak{a} \Pi T_1 = \Op(c)
\]
defines an analytic symbol $c$. Let us describe the nature of the operator $\Sigma:\mathfrak{a} \mapsto c$. In the classical Bergman-Segal case, this is the heat flow taken at a time proportional to $h$. In Remarque 4.6, \cite{Singularite-analytique-microlocale}, Sj\"ostrand considered this case as an application of his Théorème 4.5. We can generalize this:
\begin{lemma}\label{lemma:inverse-FIO}
The map $\Sigma$ is invertible in the class of analytic symbols. 
\end{lemma}
The inverse is obtained as an FIO with complex phase, whose phase is only positive along a \emph{good contour} to use the terminology of \cite{Singularite-analytique-microlocale}. 

\begin{proof}
For this, we come back to the argument in the proof of Proposition \ref{prop:pseudos-downstairs}. Recall that we wrote $S_b T_a$ as $\Op(c)$ with an integral formula. Since $\Pi T_1$ and $S_1 \Pi$ are of the form $T_a$, $S_b$, a similar formula holds. We obtain that up to exponentially small errors, and as long as $(x,\eta)$ does not approach the boundary of $A\times B$,
\[
c(x,\eta) = \int_{A\times B} \int_A e^{\frac{i}{h}\Psi(x,\eta,\alpha,y)} b(x,\alpha)a(\alpha,y)\mathfrak{a}(\alpha) dy d\alpha. 
\]
However, we will not deal with this integral in the same way as before. Indeed, now we notice that provided $(x,\eta)$ remain in a very small $A''\times B''$ contained in a small $A'\times B'$ itself contained in $A\times B$, we can change the domain of integration without losing more than exponential errors (using the same arguments as in the proof of Proposition \ref{prop:pseudos-downstairs}). Hence up to exponentially small errors,
\[
c(x,\eta) \simeq \int_{A'\times B'}\int_A e^{\frac{i}{h}\Psi(x,\eta,\alpha,y)} b(x,\alpha)a(\alpha,y)\mathfrak{a}(\alpha) dy d\alpha. 
\]
Let us be more precise on the sense of ``small'' here: these smaller and smaller sets are balls of increasingly small radius, all centered at $\alpha_0$. 

When $\alpha=(x,\eta)$, $\Psi$ is non-degenerate stationary in $y$ at $y=x$. Since now $(x,\eta,\alpha)$ are restrained to a small domain, we can apply a stationary phase argument in $y$, to deduce that for $(x,\eta)\in A''\times B''$, 
\[
c(x,\eta)=  \underset{:= \Sigma_0(c)}\int_{A'\times B'} e^{\frac{i}{h}\Phi(x,\eta;\alpha)}q(x,\eta;\alpha)\mathfrak{a}(\alpha) d\alpha + \mathcal{O}(e^{-1/Ch}),
\]
for a new analytic phase $\Phi$ and some analytic symbol $q$. The phase $\Phi$ is critical at $(x,\eta)=\alpha$, and has non-negative imaginary part, strictly positive on the boundary of $A'\times B'$. Using the fact that $\Im \Psi\geq 0$ and $\Im \partial_y^2 \Psi >0$ at the critical point, we deduce using the fundamental lemma \cite[Lemme 3.2]{Singularite-analytique-microlocale} or \cite[Lemma 2.1]{Melin-Sjostrand-75} that
\begin{equation}\label{eq:non-degenerate-Phi}
\frac{\partial^2\Phi}{\partial(x,\eta)\partial \alpha} \text{ is invertible at $(\alpha,\alpha)$.}
\end{equation}
In particular it is legitimate to call $\Sigma$ an elliptic FIO, associated with the local complex symplectomorphism
\[
\kappa : (\alpha; -\partial_\alpha \Phi) \mapsto (x,\eta; \partial_x \Phi, \partial_\eta \Phi).
\]

According to Théorème 4.5 from \cite{Singularite-analytique-microlocale}, we can invert $\Sigma$ in the class of complex phase FIOs. More precisely, since $\Im \Phi \geq 0$, there exists $\Xi$ a complex phase FIO such that $\Xi \Sigma = \Sigma\Xi=1$ in $H_{0,\alpha_0}$. Here
\begin{itemize}
	\item $H_{0,\alpha_0}$ is the space of germs of holomorphic functions near $\alpha_0$ with subexponential growth in $h$, modulo exponentially small remainders. This corresponds with the space of germs formal analytic symbols at $\alpha_0$.
	\item There exists a contour $\Gamma$ passing through $\alpha_0$ and an analytic symbol $m$ such that $\Xi$ is given by
	\[
	\Xi c (\alpha) = \int_{\Gamma} e^{-\frac{i}{h}} m(x,\eta,\alpha) c(x,\eta)dxd\eta. 
	\]
	When $c$ is only defined on a ball $B_\Gamma$ near $\alpha_0$ in $\Gamma$, we replace the integration over $\Gamma$ by integration over $B_\Gamma$. It only produces an exponentially small uncertainty on the definition of $\Xi c$ which is not important because we are working in the space $H_{0,\alpha_0}$.
\end{itemize}
\end{proof}

Equipped with the pseudo-inverse for $\Sigma$, we come back to our problem. Let us write $S_1 P T_1 = \Op(c)$, with $c$ a fixed analytic symbol on $A\times B$. In any sufficiently small sets $A'\times B'$, we can build a pseudo-inverse ${\Xi}$ to $\Sigma$ on $\mathcal{C}(A'\times B', B(0,\epsilon))$ using Lemma \ref{lemma:inverse-FIO}. We determine thus $\mathfrak{p}_0$ a symbol on $A'\times B'$ such that acting on $\mathcal{C}(A',B')$
\[
S_1 P T_1 = S_1\Pi \mathfrak{p}_0 \Pi T_1 + \mathcal{O}(e^{-1/Ch}). 
\]
i.e
\[
S_1( P- \Pi \mathfrak{p}_0 \Pi) T_1 = \mathcal{O}(e^{-1/Ch}). 
\]
Now, $P-\Pi\mathfrak{p}\Pi$ has kernel in the form $e^{i\phi/h}w$, and the operator in the LHS above is $\Op(q)$. The relation implies that $w+\mathcal{O}(h)$ on the diagonal equals $q$. Additionally, we have seen that $w$ satisfies some transport equations, and we deduce that $w= \mathcal{O}(h)$. We can reiterate this argument and deduce that the formal symbol associated with $w$ vanishes completely, i.e in $\mathcal{C}(A'\times B')$,
\[
P = \Pi \mathfrak{p}_0 \Pi  + \mathcal{O}(e^{-1/Ch}).
\]
This argument proves that for any $\alpha_0\in U$, there exists a complex neighbourhood $\alpha_0\in A'\times B'$ so that the formal symbol $\mathbf{p}$ is an analytic formal symbol in $A'\times B'$. It must thus be an analytic formal symbol on the whole $U$. One may thus pick a global realization, and end the proof. 
\end{proof}

\section{Local conjugation}

\subsection{Two digressions}

In this section, we discuss some related topics and some points that arise but are not directly related to the rest of the paper.

\subsubsection{True projectors: line bundles}

So far we have only discussed projectors modulo an exponential error. Let us discuss how one can obtain a true projector. Our arguments were local, so we need to introduce the global structure of the set $U$. There are three main situation that come to mind
\begin{enumerate}
	\item $U\subset \R^{2n}$ is an open set with smooth boundary, say bounded.
	\item $U = T^\ast M$ is a cotangent bundle of a compact manifold. 
	\item $U$ is a (closed) symplectic manifold. 
\end{enumerate}

From the existence of the phase function, we deduce that on $U$, the symplectic form is exact. This would be a strong condition in the case of $U$ being compact symplectic without boundary. However, to understand actual projectors in case (1), we would need to describe boundary conditions, and in case (2) we would need to introduce symbol classes to deal with $\xi\to\infty$ in $T^\ast M$. Both situations lead to potentially complicated questions, that are (much) beyond the scope of this article. 

Let us thus discuss a slight extension of our results, to deal with the case of positive line bundles. At the start of \S \ref{sec:study-phase}, we observed that one can always add to $\phi$ a coboundary of the form $f(\alpha)-f(\beta)$, with $f$ a real valued function, and not change the essential properties of $\phi$. Let us consider now $X$ a compact manifold, covered by open sets $U_j$, so that each $U_j$ is endowed with an analytic projector phase $\phi_j$, and so that on $U_j\cap U_k$, a real valued, real analytic function $f_{jk}$ is defined satisfying
\[
\phi_j(\alpha,\beta) = \phi_k(\alpha,\beta) + f_{jk}(\alpha)- f_{jk}(\beta). 
\]
Then the symplectic form $\omega|_{U_j}=\omega_j$ defined by $\phi_j$ coincides with the one defined by $\phi_k$, so that $(X,\omega)$ is a symplectic manifold. Denote $\lambda_j = d_\alpha {\phi_j}|_{\alpha=\beta}$. The relation given on $\phi_j$ and $\phi_k$ suggests that instead of having an operator acting on functions on $X$, we should have an operator acting on sections of a complex line bundle. More precisely, if $L\to X$ is a line bundle, and we have a section $\sigma$ of $L$ supported in $U_j\cap U_k$, the sections
\[
\mu_{j,k} = \int_{U_{j,k}} e^{\frac{i}{h} \phi_{j,k}(\alpha,\beta)} \sigma_{j,k}(\beta) d\beta,
\]
define the same section in $U_j\cap U_k$ if the transition function from $U_j$ to $U_k$ is $e^{\frac{i}{h} f_{jk}}$, i.e
\[
\sigma_j = e^{\frac{i}{h}f_{jk}}\sigma_k. 
\]
Since these transition function have modulus $1$, $L$ should be a hermitian complex line bundle. Our data also enable us to define a connection, by setting that in each $U_j$,
\[
i h \nabla 1 = \lambda_j. 
\]
The curvature of this connection is 
\[
\Omega = \frac{1}{ih}d\lambda_j = \frac{\omega}{ih},
\]
so that the first Chern class is
\[
-\frac{1}{h}[\omega]\in H^2(X,\R). 
\]
According to classical results, we thus must have $[\omega] \in h H^2(X,\Z)$. Conversely, provided $[\omega]\in h H^2(X,\Z)$, Weil's Theorem ensures that we can build exactly one such line bundle $L=L_{h}$, up to isomorphism\footnote{If this quantization condition is satisfied for some $h_0$, with bundle $L=L_{h_0}$, then it is also satisfied for $h = h_0/k$, $k\in \N^\ast$, with bundle $L_h = L^{\otimes k}$.}. In that case, we can define $e^{i\phi/h}$ as a section of 
\[
L \boxtimes L^\ast,
\]
defined in a neighbourhood of the diagonal. It then makes sense to consider operators on sections of $L$ whose kernel takes the form
\[
e^{\frac{i}{h}\phi} a, 
\]
with $a$ a function on $X\times X$ supported near the diagonal. These algebraic precautions taken, the arguments from \S \ref{sec:projector} apply, and we can find (many) operators $\Pi_0$ whose kernel takes the form above, satisfying $\Pi_0^2 = \Pi_0 + \mathcal{O}(e^{-1/Ch})$. The kernel of the operator $\Pi_0$ is analytic near the diagonal, and exponentially small away from the diagonal, so we can solve a $\overline{\partial}$ problem to find another operator $\Pi_1$, differing from $\Pi_0$ by a $\mathcal{O}(e^{-1/Ch})$ kernel, whose kernel is analytic on $X\times X$. When $h$ is small, the relation $\Pi_1^2 = \Pi_1 + \mathcal{O}(e^{-1/Ch})$ ensures that the spectrum of $\Pi_1$ is contained in two small disks centered at $0$ and $1$ respectively. For $z$ not too close to $0$ or $1$, one finds that 
\[
(z-\Pi_1)\left(\frac{1}{z} + \frac{1}{z(z-1)}\Pi_1 \right) = 1 + \frac{1}{z(z-1)} ( \Pi_1 - (\Pi_1)^2). 
\]
The right hand side can be inverted using Neumann series, and we find that this inverse is of the form $1+ \mathcal{O}(e^{-c/h})$, with the error having analytic kernel (because it is analytically smoothing). Using some spectral theory, it is then possible to finally find a true projector $\Pi$, whose kernel is analytic, and exponentially close to that of $\Pi_1$ and thus to $\Pi_0$.

\subsubsection{From pseudors to Tœplitz}

It is a well known fact that the FBI transform relates Tœplitz operators --- of the form $\Pi a \Pi$ --- with pseudodifferential operators. More precisely, given a FBI transform $T$ and its adjoint $S$, if $a$ is a ($C^\infty$) symbol, there exist $b$ and $c$ other such symbols such that
\[
S (\Pi a \Pi) T = \Op(b) + \mathcal{O}(h^\infty),\quad \Op(a) = S \Pi c \Pi T + \mathcal{O}(h^\infty).
\]
This is not very precise, and one would have to specify the context ($\R^n$ or a manifold ?), symbol classes and which FBI transform we are dealing with exactly to obtain a formal statement. Multiple such instances exist in the litterature. However, in the analytic case, while certainly the construction of $b$ is written down in several places, we could not locate the construction of $c$, that was the object of the end of \S \ref{sec:projector}. We record it here
\begin{theorem}
Let $a$ be a analytic symbol on $A\times B$ and $T$, $S$ a pair of dual FBI transforms, and $\Pi$ a projector as in \S \ref{sec:projector}. Then there exists $c$ an analytic symbol such that acting on $\mathcal{C}(A, B)$,
\begin{equation}\label{eq:toeplitz-to-Op}
\Op(a) = S \Pi c \Pi T + \mathcal{O}(e^{-1/Ch}). 
\end{equation}
\end{theorem}

\begin{proof}
Recall from \S\ref{sec:projector} that the map $c\mapsto a$ is denoted $\Sigma$. Also recall that Théorème 4.5 \cite{Singularite-analytique-microlocale} provided us (Lemma \ref{lemma:inverse-FIO}) with a twosided inverse $\Xi$ to $\Sigma$ in the space of germs of formal analytic symbol.

Let us consider $a$ an analytic symbol. We can one the one hand use the $C^\infty$ theory to show that there exists locally and globally at most one formal symbol $\mathbf{c}$ solving \eqref{eq:toeplitz-to-Op}. On the other hand, we can use ${\Xi}$ to construct an analytic symbol solving locally the problem, proving that the unique global formal solution is actually \emph{analytic} formal.
\end{proof}

\subsection{Local total reduction and the Boutet-Guillemin construction}

We mentioned above that the usual construction of the $\zeta_j$ yield pseudors that are not analytic stricto sensu: they have a formal expansion $\sum h^k \zeta_j^k$ where each $\zeta_j^k$ is analytic, but the growth as $k\to+\infty$ is not controlled. The original construction is \cite{Boutet-Guillemin} is done in the $C^\infty$ context, where the convergence problem is not relevant. We will in this section present a way to in some sense upgrade \S A.5 of \cite{Boutet-Guillemin} to the analytic category (albeit only locally).

\subsubsection{A first construction} 

We present here an argument that was communicated to us by J. Sj\"ostrand. 

Replacing $T$ by $\Pi T$, we can always assume that $\Pi T = T$. Now, we try to build the $\zeta_j$'s as annihilators of $T$. We observe that since
\[
\psi(\alpha,x) = \langle \alpha_x- x, \alpha_\xi\rangle + \mathcal{O}(\alpha_x - x)^2, 
\]
then freezing the value of $\alpha_x$ gives an elliptic analytic FIO with positive phase:
\[
U_{\alpha_x}: f\mapsto  Tf(\alpha_x, \cdot). 
\]
Proceeding as in \S 4 of \cite{Singularite-analytique-microlocale} or \S \ref{sec:projector}, we can build a microlocal inverse $V$, and then observe that
\[
(\frac{h}{i}\partial_{\alpha_x} U_{\alpha_x})
\]
is also an analytic FIO, with same phase as $U_{\alpha_x}$, so that using $V$, we can find a vector-valued analytic pseudo $P_{\alpha_x}$ such that
\[
(\frac{h}{i}\partial_{\alpha_x} - P_{\alpha_x}) T f = \mathcal{O}(e^{-1/Ch}). 
\]
This suggests to let
\[
\zeta_j = \frac{h}{i}\partial_{(\alpha_x)_j} - (P_{\alpha_x})_j.
\]
However, this rough construction yields operator that do not have good commutation properties, because a priori, we only get
\[
\left[\zeta_j, \zeta_k \right] = \mathcal{O}(h).
\]
Unfortunately, we could not find a way to improve this (beyond $\mathcal{O}(h^\infty)$). The construction is also local, and we did not find a way to globalize it.

\subsubsection{Local total reduction: the general case}

To understand exactly the structure of the ideal of operators that vanish on our projectors, it is convenient to conjugate it to a standard projector, for which this ideal is particularly easy to describe. There are several different such model projectors, each one with its particular interest. 

There is one involutive manifold for which it is very clear what the $\zeta_j$ should be, and it is the flat one, given by
\[
\mathcal{J}=  \{ (x, 0, x', \xi')\ |\ x,x',\xi'\in\R^n\}. 
\]
Indeed, this suggest to use the operators
\[
\zeta_j = h\frac{\partial}{\partial x_j}.
\]
We also need another, transverse, involutive manifold, which will be now
\[
\mathcal{J}^\ast = \{ (0, \xi, x', \xi')\ |\ \xi, x', \xi' \in \R^n\},
\]
associated with the operators
\[
\zeta_j^\ast = x_j.
\]
Both these involutive manifolds intersect transversely at the symplectic manifold
\[
\Sigma = \{ (0,0,x',\xi') | x',\xi'\in\R^n\}.
\]
The projector $\Pi^{stan}$ associated with these manifolds and operators should satisfy that $\Pi^{stan} f$ is independent of $x$ for any $f$\footnote{Probably the operators considered in \cite{Charles-2022} corresponds to projecting on certain families of polynomials in $x$ instead of constant functions. The natural question would to wonder whether they can be conjugated to one of
\[
\Pi^{(m)}f(x,x') = \frac{1}{(2\pi h)^n} \int_{\R^{4n}} e^{-\frac{i}{h}\langle y,\xi\rangle} (1+ \dots + \frac{1}{m!}(\frac{i x\xi}{h})^m) f(y, x') dyd\xi. 
\]
One key point in \cite{Charles-2022} is to introduce some vector bundles to keep track of the principal symbol of such operators in a more convenient fashion than this singular $1+ \dots + \frac{1}{m!}(\frac{i x\xi}{h})^m$.}, and does not depend on anything else than the value of $f$ at $x=0$. It must be
\[
\Pi^{stan} f(x,x') = f(0, x').
\]
Let us rewrite this into FIO form, to understand what perturbations will be possible. Certainly, 
\[
\Pi^{stan} f (x,x') = \frac{1}{(2\pi h)^{2n}} \int_{\R^{4n}} e^{\frac{i}{h}(-\langle y,\xi\rangle + \langle x'-y',\xi'\rangle)} f(y, y') dy dy' d\xi d\xi'.
\]

If we are to conjugate locally any microlocal projector $\Pi$ as studied in Part \ref{part:1}, we will need to simultaneously transform $\mathcal{J}_\phi$ into $\mathcal{J}$ and $\mathcal{J}_\phi^\ast$ into $\mathcal{J}^\ast$. According to Lemma \ref{lemma:flattening-pair} it is possible, using a complex symplectomorphism $\kappa$.

Let now $\Pi_\phi$ be a microlocal projector with phase $\phi$ as in Section \ref{sec:projector}. Let $U$ be an analytic FIO quantizing the $\kappa$ from Lemma \ref{lemma:flattening-pair}, microlocally elliptic near $(\alpha,\vartheta_\alpha)$. Since $U$ is associated with a complex symplectomorphism that does not satisfy any type of positivity assumption, it is not an FIO whose phase has positive imaginary part. One must thus use good contour in its definition, in the spirit of Théorème 4.5 from \cite{Singularite-analytique-microlocale}; we will mostly gloss over these details as we have already recalled the main ideas in the proof of Lemma \ref{lemma:inverse-FIO}. The only relief we can offer to the reader is that since the symplectomorphism maps real points of $\Sigma_\phi$ to real points of $\Sigma$, we can take the phase real on $\Sigma$. 

Consider $\Pi' = U \Pi_\phi U^{-1}$. It must be associated with the relation
\[
(0,\xi,x',\xi') \mapsto (x, 0, x', \xi'). 
\]
In other words, we can write $\Pi'$ microlocally near $0$ as
\[
\frac{1}{(2\pi h)^{2n}} \int_{\Gamma} e^{\frac{i}{h}(-\langle y,\xi\rangle + \langle x'-y',\xi'\rangle)}a(x,y,\xi,x',y',\xi') f(y, y') dy dy' d\xi d\xi'.
\]
Observe that we are integrating over a good contour $\Gamma\subset \C^{4n}$, so that the boundary of $\Gamma$ produces exponentially small contributions. This contour is close to $\R^{4n}$ near $0$. In particular this formula makes sense for $f$ analytic, with a sufficiently large radius of convergence.

The amplitude here has too many variables, we can eliminate some of them. First, we can eliminate $y'$ because the operator behaves as a pseudo in the $x'$ variable. The behaviour of the operator should not depend too much on the behaviour of $f$ for large values of $y$, so we try expanding the amplitude in powers of $y$. Integrating by parts in $\xi$, we obtain that
\[
a=\sum \frac{1}{k!} y^k \partial_y^k a(x,0,\xi, x',\xi') \sim \tilde{a} =\sum \frac{(h/i)^k}{k!} \partial_\xi^k\partial_y^k a(x,0,\xi, x',\xi').
\]
Direct estimation shows that if $a$ is an analytic symbol, so is $\tilde{a}$, so that we can eliminate $y$ from the amplitude. The form of the operator is actually
\[
\Pi' \simeq\frac{1}{(2\pi h)^{2n}} \int_{\Gamma} e^{\frac{i}{h}(-\langle y,\xi\rangle + \langle x'-y',\xi'\rangle)}a(x,\xi,x',\xi') f(y, y') dy dy' d\xi d\xi',
\]
with $a$ an analytic symbol, elliptic near $0$.

So far we have only used that $\Pi'$ is an FIO. Using the fact that $\Pi'$ is a projector, we want to reduce further the operator. For this, let us study the action of pseudo-differential operators on $\Pi^{stan}$. We will use the local product decomposition. I.e we observe that $\Pi^{stan}$ writes as
\[
\Pi^{stan} = T S, 
\]
where, for $f$ a function of $x'$, 
\[
Tf (x,x') = f(x'), 
\]
and for $g$ a function of $(x,x')$, 
\[
Sg(x') = g(0,x').
\]

Now, for a symbol $a(x,\xi,x',\xi')$, we observe by direct computation that,
\begin{align*}
\Op(a) T &= \Op(a(x,0, x',\xi')) T,\\
S \Op(a) &= S \Op(a(0,\xi,x',\xi')),\\
 S\Op(a) T &= S\Op(a(0,0,x',\xi'))T.
\end{align*}
If we have a symbol $a(x',\xi')$, we can interpret it as a function of $(x,\xi,x',\xi')$, with trivial dependence on $(x,\xi)$. With that convention,
\begin{equation}\label{eq:commutation}
\Op(a(x',\xi')) T = T \Op(a(x',\xi')).
\end{equation}
(and likewise for $S$).

\begin{lemma}\label{lemma:local-projectors-around-Pi-stan}
Let $a,b$ be elliptic analytic symbols defined in a neighbourhood of $0$. Then the following are equivalent
\begin{enumerate}
	\item $\Op(a) \Pi^{stan} \Op(b)$ is a projector near $0$
	\item Near $0$, $\Op(b)^{-1} \Op(a) = \Op(c)$ with 
	\[
	c = 1 + \mathcal{O}( x, \xi).
	\]
	\item There exists an elliptic analytic symbol $\sigma$ such that
	\[
	\Op(a) \Pi^{stan} \Op(b) = \Op(\sigma) \Pi^{stan} \Op(\sigma)^{-1}
	\]
\end{enumerate} 
\end{lemma} 

\begin{proof} Let us assume that for some elliptic symbols $a,b$,
\[
\Op(a) \Pi^{stan} \Op(b)
\]
is a projector. Conjugating by $\Op(b)$, we can and will assume that $b=1$. The condition becomes
\[
\Op(a) \Pi^{stan} \Op(a) \Pi^{stan} = \Op(a) \Pi^{stan}. 
\]
In particular, this gives
\[
\Big[ \Pi^{stan} \Op(a) \Pi^{stan} \Big]^2 = \Pi^{stan} \Op(a) \Pi^{stan}.
\]
This is exactly
\[
\Op(a(0,0,x',\xi'))^2 = \Op(a(0,0,x',\xi')). 
\]
Since $a$ is elliptic, this implies that $a(0,0,x',\xi') = 1$. Reciprocically, if $a(0,0,x',\xi') = 1$, we must have $\Pi^{stan} \Op(a) \Pi^{stan}= \Pi^{stan}$, so that $\Op(a) \Pi^{stan}$ is indeed a projector.

Now, if we consider the formula for $\Op(a) \Pi^{stan}$, we observe that it makes no difference to replace $a$ by $\tilde{a}(x,\xi,x',\xi') = a(x,0,x',\xi')$. In particular, we may as well assume that $a$ does not depend on $\xi$. However in this case, we also have the identity
\[
\Pi^{stan} \Op(\tilde{a}) = \Pi^{stan}\Op(a(0,0,x',\xi')) = \Pi^{stan}. 
\]
Since $a$ is elliptic, we may invert it and find that 
\[
\Pi^{stan} \Op(\tilde{a})^{-1} = \Pi^{stan}. 
\]
We have thus
\[
\Op(a) \Pi^{stan} = \Op(\tilde{a}) \Pi^{stan} \Op(\tilde{a})^{-1}. 
\]
\end{proof}

Let us come back to $\Pi'$. Composing $\Pi'$ with $T$ gives
\[
\Pi' T f \simeq \frac{1}{(2\pi h)^{2n}} \int_{\Gamma} e^{\frac{i}{h}(-\langle y,\xi\rangle + \langle x'-y',\xi'\rangle)}a(x,\xi,x',\xi') f(y') dy dy' d\xi d\xi'.
\]
By stationary phase, this is 
\[
\Pi' T f(x,x') \simeq \frac{1}{(2\pi h)^{n}} \int_{\Gamma} e^{\frac{i}{h}\langle x'-y',\xi'\rangle}a_T(x,x',\xi') f(y') dy dy' d\xi d\xi',
\]
for some new symbol $a_T$, which is also elliptic near $0$. This is nothing else than
\[
\Op( a_T ) T f, 
\]
where $a_T(x,\xi,x',\xi') = a_T(x,x',\xi')$ by abuse of notation. Proceeding likewise for $S\Pi'$, we obtain that for some elliptic symbol $a_S$
\[
S \Pi' \simeq S \Op(a_S) .
\]

Now, recall that $S \Pi' T$ must be an elliptic pseudo $P(x',h\partial_{x'})$ that we can invert. Now, we consider
\[
\Pi'':=\Pi' T P^{-1} S \Pi' = \Op(a_T) P^{-1} \Pi^{stan} \Op(a_S). 
\]
(the second equality follows from \eqref{eq:commutation}). One can check directly that it is an approximate projector. From the transport equation argument on the amplitudes used in the proof of Proposition \ref{prop:construction-many-projectors}, since $\Pi'\Pi'' = \Pi'' \Pi' = \Pi''$, we deduce that actually, 
\[
\Pi'' \simeq \Pi'
\]
In particular, Lemma \ref{lemma:local-projectors-around-Pi-stan} applies to $\Pi'$. 

We can sum up the results of this section so far as follows
\begin{theorem}\label{thm:local-structure-projectors}
Let $\phi$ be an analytic projector phase on a star shaped open set $U$, and $\alpha \in U$. We can find $\zeta_1,\dots,\zeta_n$ and $\zeta_1^\ast,\dots,\zeta_n^\ast$ analytic pseudors microlocally defined near $\Sigma_\phi$, an analytic approximate microlocal projector $\Pi$ with phase $\phi$, so that for any other elliptic analytic approximate microlocal projector $\Pi'$ on $U$ with same phase, there exists an analytic pseudor $A$ microlocally defined near $\Sigma_\phi$ with $A$ elliptic, so that 
\begin{align*}
\zeta_j \Pi &= \Pi \zeta^\ast_j = 0,\quad j=1\dots n\\
[\zeta_j,\zeta_k] &= [\zeta_j^\ast,\zeta_k^\ast] = 0,\\
[\zeta_j,\zeta_k^\ast] &=  h \delta_{jk}\\
\Pi' &= A \Pi A^{-1}. 
\end{align*}
(the equalities here are taken microlocally near $\Sigma_\phi$).

If $\Pi'$ is an analytic approximate microlocal projector with a different phase on the same set $U$, then the same statement holds, except that $A$ has to be taken in the class of elliptic FIOs with complex phase. 
\end{theorem}
Observe that the $A$ is not unique.

\subsubsection{The self-adjoint model case}

The drawback of the previous discussion is that it does not lend itself to a more precise analysis in the self-adjoint case. It also uses FIOs with non-positive phase which are not very practical to use outside of spaces of germs of analytic symbols. We thus study this almost independently. The most famous self-adjoint local model\footnote{We could probably make $\Pi^{stan}$ to be self-adjoint, if we used a different notion of duality. For example associated with the bracket $\langle f , g\rangle = \int f \overline{\hat{g}}$, but we digress.} is the following. Consider
\[
H= L^2(\C^n, e^{-|z|^2/h} dz), 
\]
and the Bargmann-Segal projector
\[
\Pi^{BS}f(z)= \frac{1}{(\pi h)^{n}} \int_{\C^n} e^{\frac{1}{h} \langle z, \overline{z'}\rangle - \frac{|z'|^2}{h}} f(z') dz'.
\]
It is the orthogonal projector on the holomorphic functions in $H$. We can conjugate it to a projector on $L^2(\C^n, dz)$, to obtain the phase 
\begin{align*}
\phi^{BS} 	&= - i \langle z, \overline{z'}\rangle + \frac{i}{2}(|z|^2 + |z'|^2) \\
			&=\Im(\langle z, \overline{z'}\rangle) +  \frac{i}{h}|z-z'|^2.
\end{align*}
In the imaginary part term we recognize the standard real symplectic form on $C^n \simeq \R^{2n}$. 

However, there is a way to construct a self-adjoint projector closer to $\Pi^{stan}$ from above. Indeed, we can decompose $\Pi^{stan}$ as a tensor product of the identity in the $x'$ variable, and a rank one projector in the $x$ variable:
\[
\Pi^{stan} f = | 1 \rangle \langle \delta_0 |_{x} \otimes \mathbf{1}_{x'}
\]
To obtain a self-adjoint projector, it suffices to replace $1$ by $u\in L^2_x$ with norm $1$, and construct
\[
\Pi_u f = |u \rangle \langle u|_{x} \otimes \mathbf{1}_{x'}.
\]
For this to be microsupported on $x=0$, the most natural choice is
\[
u = (\pi h)^{-n/2} e^{- \frac{x^2}{2h}}. 
\]
We thus get the projector\footnote{Of course, this is but a simple transformation away from being the Bargmann-Segal projector.}
\[
\Pi^{\widetilde{stan}} f(x,x') = \frac{1}{(\pi h)^n} \int_{\R^n} e^{- \frac{x^2 + y^2}{2h}} f(y,x')dy.
\]
The corresponding new operators are
\[
\zeta_j =\frac{1}{2}( h\partial_{x_j} - x_j),\quad \zeta_j^\ast  = - \frac{1}{2}( h\partial_{x_j} + x_j).
\]
We observe that here $\zeta_j^\ast$ really is the adjoint of $\zeta_j$, as the notation suggested. We also have the decomposition
\[
\Pi^{\widetilde{stan}} f = T_u T_u^\ast, 
\]
where
\[
T_u f(x,x') = u(x) f(x'). 
\]

Now, let us come back to the original normal form problem, and let $\Pi$ be a self-adjoint projector as before. According to Proposition \ref{prop:normal-form-positive-involutive}, we can find a real symplectomorphism $\kappa$ flattening $\Sigma$, and transforming the corresponding involutive manifold into the standard positive one. Let us quantize $\kappa$ into $U$ a unitary FIO. By assumption, we find that $U\Pi U^\ast$ is a self-adjoint projector FIO whose phase is the same as $\Pi^{\widetilde{stan}}$. Since the result is a projector, we can use Lemma \ref{lemma:local-projectors-around-Pi-stan} to deduce that
\[
U\Pi U^\ast  = \Op(a) \Pi^{\widetilde{stan}} \Op(a)^{-1},
\]
for some elliptic analytic symbol $a$. Now, both sides of the equation are self-adjoint, so that, with $A=\Op(a)$,
\[
A^\ast A \Pi^{\widetilde{stan}} = \Pi^{\widetilde{stan}} A^\ast A. 
\]
We can extend this identity to powers of $A^\ast A$, and then by analytic functional calculus, to $P= (A^\ast A)^{1/2}$ (the positive square root determination, which is also a pseudor). Then we have the polar decomposition: 
\[
A = B P, 
\]
with $B$ unitary, and
\[
A \Pi^{\widetilde{stan}} A^{-1} = B P \Pi^{\widetilde{stan}} P^{-1} B^\ast = B \Pi^{\widetilde{stan}} B^\ast.
\]
We have proved:
\begin{corollary}\label{cor:decomp-self-adjoint}
With the notations of Theorem \ref{thm:local-structure-projectors}, if the phase $\phi$ is self-adjoint, we can assume that $\Pi$ is self-adjoint, and if $\Pi'$ is also self-adjoint, then the conjugation is by unitary pseudors (or unitary FIOs if the phase are different). We can also assume that $\zeta_j^\ast$ is the adjoint of $\zeta_j$.
\end{corollary}

\section{Involutive manifolds in transverse intersection}

In this section are gathered some relevant facts about involutive manifolds that were used in the proofs above. Most of them are probably classical, but are still convenient to recall. We will work in the holomorphic complex symplectic case (i.e ``smooth'' means holomorphic). When we write ``real symplectic'', we mean to work with real analytic real symplectic geometry, but will work with holomorphic extensions all the time. 

Some statements make sense in a purely real case, and in this case the proofs in the $C^\infty$ real symplectic case are almost identical. In the case of almost analytic involutive manifolds, one has modify slightly some notations, but the arguments are also essentially the same. 

\subsection{Involutive manifolds}

\begin{definition}
Let $(X,\omega)$ be a $2n$-dimensional symplectic manifold. Let $\mathcal{J}\subset X$ be a smooth submanifold. Then $\mathcal{J}$ is said to be involutive if
\[
(T\mathcal{J})^\perp \subset T\mathcal{J}.
\]
In other words, whenever for every $v\in T\mathcal{J}$, $\omega(u,v)=0$, then $u\in T\mathcal{J}$. By considerations of dimension, $\mathcal{J}$ must have dimension $n\leq p \leq 2n$. If $p=n$, $\mathcal{J}$ is Lagrangian and $T\mathcal{J}^\perp = T\mathcal{J}$. Coisotropic is a synonym for involutive. 

Let $(V,\omega)$ be a manifold endowed with a closed $2$-form with constant rank. $V$ is said to be presymplectic. The data $(\mathcal{J}, \omega_{|\mathcal{J}})$ form a presymplectic manifold if $\mathcal{J}$ is an involutive manifold.
\end{definition}

\begin{proposition}\label{prop:union-lagrangian}
Let $\mathcal{J}\subset X$ be a smooth submanifold. Assume that $\mathcal{J}$ is a union of lagrangian manifolds. Then $\mathcal{J}$ is involutive. 
\end{proposition}

\begin{proof}
Let $x\in \mathcal{J}$. Then there exists $L\subset \mathcal{J}$ a lagrangian manifold passing through $x$. We thus have
\[
T_x L \subset T_x \mathcal{J},
\]
and
\[
T_x \mathcal{J}^\perp \subset T_x L^\perp = T_x L \subset T_x \mathcal{J}.
\]
\end{proof}

According to Weinstein's embedding theorem, any lagrangian manifold $L$ has a neighbourhood that is symplectomorphic to a neighbourhood of the zero section in $T^\ast L$. The generalization to involutive manifolds is due to Gotay \cite{Gotay-1982}
\begin{theorem}[The co-isotropic embedding theorem]
Let $(V,\omega)$ be a presymplectic manifold. Then there exists a symplectic $(X,\omega')$ and an embedding $p:V \to X$ such that $p(V)$ is involutive and, as presymplectic manifolds,
\[
(V,\omega) \simeq (p(V), \omega'_{|V}).
\]
Additionally, for any other embedding $p' :V\to X'$, with involutive image, there exists a symplectomorphism $\kappa$ from a neighbourhood of $p(V)$ to a neighbourhood of $p'(V)$, such that $\kappa\circ p = p'$. 
\end{theorem}

Let us now describe the local structure of involutive manifolds.
\begin{definition}
Let $(V,\omega)$ be a presymplectic manifold. Then the bundle $\ker \omega$ is integrable, and defines the so-called \emph{null-foliation} $\mathcal{F}$ of $V$. Additionnally, if $\Sigma_1$ and $\Sigma_2$ are two local transversal to $\mathcal{F}$, and $\kappa:\Sigma_1\to \Sigma_2$ is a local holonomy along $\mathcal{F}$, then both $(\Sigma_1,\omega_{|\Sigma_1})$ and $(\Sigma_2,\omega_{|\Sigma_2})$ are symplectic, and $\kappa$ is a symplectomorphism.
\end{definition}

\begin{proof}
Since $\omega$ has constant rank, $\ker \omega$ is a smooth vector bundle. Let $X,Y$ be two vector fields contained in $\ker \omega$. Then
\[
\mathcal{L}_X \omega = \imath_X d\omega + d(\imath_X \omega) = 0,
\]
and likewise for $Y$, so that
\[
\imath_{[X,Y]} \omega = [\imath_X, \mathcal{L}_Y]\omega = 0.
\]
The bundle is thus integrable. 

The restriction of $\omega$ to $\Sigma_1$, must be closed, and since $\Sigma$ is complementary to $\ker \omega$, $\omega$ is non-degenerate on $T\Sigma_1$. We can realize $\kappa$ as the flow at time $1$ of a vector field $X$ along the leaves of $\mathcal{F}$, and since for such a vector field, $\mathcal{L}_X \omega =0$, we must have $\kappa^\ast \omega =\omega$.
\end{proof}

One easy corollary of the Weinstein neighbourhood theorem is that given a smooth lagrangian manifold $L$, the space of lagrangian manifolds that is $C^1$ close to $L$, modulo hamiltonian diffeomorphisms is isomorphic to a neighbourhood of zero in the cohomology group $H^1(L)$. In the case of involutive manifolds, the situation is much more involved. Indeed the global leaf structure of $\mathcal{F}$ has a role to play. In the particular case that the null foliation has compact leaves and $\pi : \mathcal{J}\to \mathcal{J}/\mathcal{F}$ is a smooth fibration, and one deforms \emph{in this class}, W. D. Ruan \cite{Ruan-2005} proved that the deformation space modulo hamiltonian equivalence is locally a smooth finite dimensional manifold. However in the general case, the deformation space may be quite wild, and \emph{not} a manifold. For references, we point to \cite{Oh-Park-2005} and \cite{Zambon-2008}. Since in the case that is of most interest to us the leaves of $\mathcal{F}$ will be contractible, we will be in the ``trivial'' case of this deformation theory. 

\begin{proposition}\label{prop:normal-form-1-involutive}
Let $\mathcal{J}\subset X$ be an involutive manifold of codimension $p$. Near any point $\alpha\in \mathcal{J}$, we can find Darboux coordinates $(x,x',\xi,\xi')$ such that 
in these coordinates $\mathcal{J}$ coincides with
\[
\mathcal{J}_p = \{ (x,x',\xi,0) \ |\ x,\xi\in \C^{n-p},\ x'\in\C^p\}\subset \C^{2n}.
\]
\end{proposition}

\begin{corollary}\label{cor:local-annihilation}
Involutive manifolds of $X$ of codimension $p$ are exactly those submanifolds $\mathcal{J}$ of $X$ such that close to any point $\alpha\in X$, one can find locally defined smooth functions $\zeta_1, \dots, \zeta_p$ satisfying
\begin{enumerate}
	\item $\mathcal{J}$ coincides with $\{\beta\ |\ \zeta_1(\beta)=\dots= \zeta_p(\beta) =0\}$.
	\item $d\zeta_1$, \dots, $d\zeta_p$ are independent at $\mathcal{J}$.
	\item for $j,\ell=1\dots p$, $\{\zeta_j, \zeta_\ell\}_{|\mathcal{J}}=0$. 
\end{enumerate}
One may additionally assume that the brackets $\{\zeta_j, \zeta_\ell\}_{|\mathcal{J}}$ vanish everywhere, not only on $\mathcal{J}$.
\end{corollary}

We prove the corollary first:
\begin{proof}
There are two implications to prove. We consider first a set $\mathcal{J}\subset X$ that comes with such a collection of local systems $(\zeta_1, \dots, \zeta_p)$. From the independence of the $d\zeta_j$, $\mathcal{J}$ must be a smooth codimension $p$ submanifold. 

We have two notions of orthogonality: the natural one between vectors in $TX$ and covectors in $T^\ast X$, and the one between two vectors in $TX$, given by the symplectic form. The symplectic form also gives a representation of covectors as vectors, and under this representation, both notions of orthogonality are identified. We can thus write
\[
T\mathcal{J}^\perp \simeq\Span\{d\zeta_j \ |\ j=1\dots p\} \simeq \Span\{ H_{\zeta_j} \ |\ j=1\dots p\}.
\]
It remains thus to prove that the $H_{\zeta_j}$ are tangent to $\mathcal{J}$. Here the condition on the brackets ensures that for $j,\ell=1\dots p$,
\[
(H_{\zeta_j})_{|\mathcal{J}} \in \ker d\zeta_\ell,
\]
which is what we required. 

Now, in the other direction, we may take the normal form of Proposition \ref{prop:normal-form-1-involutive}, and take 
\[
\zeta_j = \xi'_j. 
\]
In particular, one sees that the brackets vanish identically in this case. 
\end{proof}

Let us now prove Proposition \ref{prop:normal-form-1-involutive}
\begin{proof}
Considering $\mathcal{J}$ as a presymplectic manifold, locally near $\alpha$, we can trivialize the null foliation, and pick Darboux coordinates on a transversal simultaneously, so that we have coordinates $(x,x',\xi)$ on $\mathcal{J}$ near $\alpha$, so that
\[
\omega_{|x'=0} = dx\wedge d\xi, 
\]
and
\[
\ker \omega = \Span\{\partial_{x'_j}\ |\ j=1\dots p\}.
\]
It follows that
\[
\omega = \omega(x,\xi,x'; dx, d\xi).
\]
Now, since $d\omega= 0$, we get
\[
\partial_{x'}\omega = 0, 
\]
so that finally $\omega = dx \wedge d\xi$. In other words, the only local invariants of presymplectic manifolds are the dimension and rank of the $2$-form. We can now use the Gotay theorem to conclude. 
\end{proof}

\begin{lemma}\label{lemma:Lagrangian-associated-with-Involutive-manifold}
Let $\mathcal{J}\subset X$ be an involutive submanifold of a symplectic manifold $X$, and let $\Sigma\subset \mathcal{J}$ be a symplectic submanifold, transverse to the null foliation $\mathcal{F}$ of $\mathcal{J}$. Assume that the global leaves of the null foliation are manifolds. Consider the set
\[
\Lambda_{\mathcal{J},\Sigma}:=\{ ( x, y)  \ | \ x\in \mathcal{F}(y),\ y\in\Sigma \} \subset X\times \Sigma.
\]
Then $\Lambda_{\mathcal{J},\Sigma}$ is a lagrangian submanifold (for the twisted symplectic form $\omega_X - \omega_\Sigma$).  
\end{lemma}

\begin{proof}
We have seen in the proof of Proposition \ref{prop:normal-form-1-involutive} that we can find local Darboux coordinates so that $\Sigma$ coincides with
\[
\{ (x,0,\xi,0) \ |\ (x,\xi)\in\C^{n-p} \}.
\]
and $\mathcal{J}$ coincides with
\[
\{ (x,x',\xi,0)\ |\ (x,\xi)\in\C^{n-p},\ x'\in \C^p\}.
\]
Then $\Lambda_{\mathcal{J},\Sigma}$ coincides with
\[
\{ (x,x',\xi,\xi',y,\eta)\ |\ x,y,\xi,\eta\in\C^{n-p},\ x',\xi'\in\C^p,\ x=y,\ \xi=\eta,\ \xi'=0\},
\]
with the symplectic form
\[
dx\wedge d\xi - dy\wedge d\eta + dx' \wedge d\xi'.
\]
This is certainly lagrangian, so that $\Lambda_{\mathcal{J},\Sigma}$ is lagrangian near $\Sigma$. 

To conclude the proof, we observe that if $X$ is a vector field supported in a small open set of $\mathcal{J}$, and tangent to $\mathcal{F}$, then writing in local coordinates
\[
X(x,x',\xi) = X_1 \frac{\partial}{\partial x'_1} + \dots + X_p \frac{\partial}{\partial x'_p}, 
\]
we see that $X$ is the restriction to $\mathcal{J}$ of the Hamilton vector field of 
\[
g(x,x',\xi',\xi')= X_1 \xi'_1 + \dots + X_p \xi'_p.
\]
For any point $(x,a)\in \Lambda_{\mathcal{J},\Sigma}$, we can thus find a symplectomorphism $\kappa$ from a neighbourhood of $x$ to a neighbourhood of $a$ with $\kappa(x)=a$, and $\kappa$ preserving $\mathcal{J}$. Since symplectomorphisms preserve lagrangian manifolds, we are done
\end{proof}

\begin{definition}
Let $\mathcal{J}$ be an involutive manifold containing $\Sigma$ symplectic, transverse to $\mathcal{F}$, and such that $\Sigma=\widetilde{\Sigma\cap\R^{2n}}$. Then we say that $\mathcal{J}$ is (strictly) positive if $\Lambda_{\mathcal{J},\Sigma}$ is itself (strictly) positive. Actually, $\mathcal{J}$ is strictly positive if and only if
\[
i \omega(v,\overline{v}) < 0
\]
for every non zero $v\in T\mathcal{F}$ above $\Sigma$.
\end{definition}

\begin{proof}
We can find  real Darboux coordinates so that 
\[
\Sigma = \{ (x,0,\xi,0)\ |\ x,\xi\in\C^{n-p}\}.
\]
We can also assume that
\[
\mathcal{J} = \{ (x,x',\xi,F(x,x',\xi))\ |\ x,\xi\in\C^{n-p},\ x'\in\C^p\},
\]
where $F(\bullet,0,\bullet)=0$ to ensure that $\Sigma\subset \mathcal{J}$. Since $\mathcal{J}$ must be involutive, we use the Lie brackets caracterisation to deduce that 
\[
\partial_{x'_j} F_k - \partial_{x'_k} F_j = \sum_{\ell} \partial_{\xi_\ell} F_j \partial_{x_\ell} F_k - \partial_{\xi_\ell} F_k \partial_{x_\ell} F_j. 
\]
i.e\footnote{Equation \eqref{eq:master-equation} is called the ``master equation'' in \cite{Oh-Park-2005} for example, and it is an instance of a Maurer Cartan equation: it can be formulated as
\[
d_{\mathcal{F}} F = \{ F, F \}.
\]}
\begin{equation}\label{eq:master-equation}
\partial_{x'} F - (\partial_{x'} F)^\top = (\partial_x F)^\top \partial_\xi F - (\partial_\xi F)^\top \partial_x F. 
\end{equation}
In particular, $\partial_{x'} F(\bullet,0,\bullet)$ is valued in symmetric matrices.

By definition, $\Lambda$ is strictly positive if for every $\alpha\in\Lambda_\R$, and every $u\in T_\alpha\Lambda$ not in $T(\Lambda_\R)\otimes \C$, $-i\omega(u,\overline{u})>0$. Such a tangent vector writes as
\[
u =( v, w, t, d_{(x,0,\xi)}F(v,w,t); v, t),
\]
so that $w\neq 0$. Since $F(x,0,\xi)=0$, this simplifies to
\[
u= (v, w, t, \partial_{x'}F(x,0,\xi)w ; v , t)
\]
Now, 
\[
\omega(u,\overline{u}) = \langle \partial_{x'}F(x,0,\xi)w, \overline{w}\rangle - \langle \overline{\partial_{x'}F(x,0,\xi)w}, w\rangle.
\]
Since $\partial_{x'} F(\bullet,0,\bullet)$ is a symmetric matrix, we conclude that
\[
\omega(u,\overline{u}) = 2 i \langle \Im (\partial_{x'} F(x,0,\xi)) w, \overline{w}\rangle.
\]
From the computation above, we also obtain that the kernel of $\omega_{|T\mathcal{J}}$ above $\Sigma$ is given by the vectors
\[
(0,w,0, \partial_{x'}F w),
\]
so we get the desired equivalence.
\end{proof}

According to Proposition \ref{prop:normal-form-1-involutive}, from the point of view of purely complex symplectic geometry, all involutive manifolds look locally the same. However, when working with positive involutive manifolds, it is natural to try to restrict to \emph{real symplectic} mappings. Such maps preserve the class of positive involutive manifolds, and we obtain:
\begin{proposition}\label{prop:normal-form-positive-involutive}
Let $\mathcal{J}$ be a strictly positive complex symplectic involutive manifold containing $\Sigma=\widetilde{\Sigma\cap\R^{2n}}$, so that 
\[
\mathcal{J}_\R = \Sigma_\R.
\]
We can find local real analytic, real Darboux coordinates in which $\Sigma$ coincides with 
\[
\{ (x,0,\xi,0)\ |\ x,\xi\in\C^{n-p}\},
\]
and $\mathcal{J}$ coincides with 
\[
\{ (x, x', \xi, i x')\ |\ x,\xi\in \C^{n-p},\ x'\in \C^p \}. 
\]
\end{proposition}

\begin{proof}
Just as above, we can find real Darboux coordinates so that $\Sigma$ coincides locally with
\[
\{ (x,0,\xi,0)\ |\ x,\xi\in\C^{n-p}\},
\]
and $\mathcal{J}$ with 
\[
\{ (x, x', \xi, F(x,x',\xi))\ |\ x,\xi\in \C^{n-p},\ x'\in \C^p \}. 
\]
Now, recall that $\Lambda_{\mathcal{J},\Sigma}$ is a lagrangian, so that we can find a phase function $\phi$ such that it is generated by
\[
\Phi:=\langle x - y, \xi\rangle + \phi(x,x',\xi). 
\]
Necessarily, $\partial_{x'}\phi(x,0,\xi)=0$ and $\phi(x,0,\xi)$ is a constant because $\Sigma\subset \mathcal{J}$, and we can assume this constant to be $0$. The positivity of $\mathcal{J}$ corresponds to $\Im \partial^2_{x'} \phi(x,0,\xi) > 0$. Let us thus form
\[
\Phi_t := \langle x-y,\xi\rangle + (1-t)\phi(x,x',\xi) + \frac{i t }{2} x'^2. 
\]
The set
\[
\mathcal{J}_t = \{ (x,x',\partial_x \Phi_t , \partial_{x'} \Phi_t)\ |\ x,y,\xi\in\C^{n-p},\ x'\in\C^p\},
\]
is a $2n - p$ dimensional manifold because $\partial_\xi \partial_x \Phi_t(x,0,\xi,y)=1$. It is a (disjoint) union of lagrangians, so it is an involutive manifold, and it is strictly positive for $t\in[0,1]$. 

We would like to find $p_t$ an analytic family of analytic function, real valued on the reals, so that the (non-autonomous) hamilton flow generated by $p_t$ at time $1$ maps $\mathcal{J}_0$ to $\mathcal{J}_1$. It is equivalent to requiring that $p$ solves the Hamilton Jacobi-type equation
\begin{equation}\label{eq:HJ-p}
\partial_t \Phi_t(x,x',\xi) = p_t(x,x',\partial_x \Phi, \partial_{x'} \Phi_t). 
\end{equation}
In other words, $p_t$ is determined on $\mathcal{J}_t$, and must take real values on the reals. Let us write
\[
\mathcal{J}_t = \{( x,x',\xi, F_t(x,x',\xi))\ |\ x,\xi\in \C^{n-p},\ x'\in\C^p\}.
\]
And also let
\[
F_t^\flat(x,x',\xi)= \overline{F_t(\overline{x,x',\xi})}.
\]
The map
\[
\psi_t:(x,\xi, x',\xi')\mapsto (x,\xi, \xi'- F_t(x,x',\xi), \xi'-F_t^\flat(x,x',\xi))
\]
is a local diffeomorphism thanks to the strict positivity of $\mathcal{J}_t$, so that we can define from $p$ another analytic function $q$ with
\[
q_t(x,\xi, \xi'-F, \xi'-F^\flat) = p_t(x,\xi,x',\xi'). 
\]
Let translate the reality condition of $p_t$ for $q_t$. 
\begin{align}
q_t(x,\xi, \xi' - F, \xi' - F^\flat) &= p_t(x,\xi,x',\xi') \notag \\
									&= \overline{p_t(\overline{x,\xi,x',\xi'})} \notag\\
									&= \overline{q_t(\overline{x}, \overline{\xi}, \overline{\xi} - F(\overline{x,x',\xi}), \overline{\xi} - F^\flat(\overline{x,x',\xi}))}, \notag \\
									&= \overline{q_t(\overline{x}, \overline{\xi}, \overline{\xi - F^\flat({x,x',\xi})}, \overline{\xi - F(x,x',\xi)})}, \label{eq:reality-condition-q}
\end{align}
This means
\[
\overline{q_t(x,\xi, \lambda, \mu)} = q_t(\overline{x,\xi,\mu, \lambda}).
\]
Additionally, the Hamilton Jacobi equation \eqref{eq:HJ-p} for $p_t$ is equivalent to the following equation on $q_t$:
\begin{equation}\label{eq:HJ-q}
q_t(x,\xi,0,\mu)= G_t(x,\xi,\mu). 
\end{equation}
Here, the function $G_t$ is determined from the knowledge of $\Phi_t$ and $F_t$. Additionally, from $\partial_t \Phi_t(x,0,\xi) =0$, we deduce that certainly.
\[
G_t(x,\xi,0)=0.
\]
Let us set
\[
\tilde{q}_t(x,\xi,\lambda,\mu) = G_t(x,\xi,\mu) + \overline{G_t(x,\xi, \overline{\lambda})}.
\]
This is well defined, and satisfies both the reality condition \eqref{eq:reality-condition-q} and the Hamilton-Jacobi equation \eqref{eq:HJ-q}.
\end{proof}

We observe that $\overline{\mathcal{J}}$ coincides locally with
\[
\{( x,x',\xi, F^\flat(x,x',\xi))\ |\ x,\xi\in C^{n-p},\ x'\in\C^p\}.
\]

We close this discussion with the following
\begin{lemma}\label{lemma:positive-lagrangian-inside-J}
Let $\mathcal{J}$ be a strictly positive involutive manifold, and let $\Lambda\subset \mathcal{J}$ be a lagrangian manifold. Then $\Lambda$ is strictly positive if and only if $\Lambda\cap \Sigma$ also is. 
\end{lemma}

\begin{proof}
One of the implications is elementary, so we concentrate on the other one. Assume that $\Lambda\cap\Sigma$ is strictly positive. Then by Proposition \ref{prop:normal-form-positive-involutive}, we can write locally
\[
\Lambda = \{ (x, x', \partial_x f(x), ix')\ |\ x\in \C^{n-p},\ x'\in\C^p \}. 
\]
we see that $\Lambda$ is strictly positive if and only if at real points of $\partial_x f$,
\[
\Im \partial_x^2 f > 0.
\]
\end{proof}

\subsection{Pairs of involutive manifolds}
Now we consider the case of two involutive manifolds intersecting transversally.
\begin{lemma}\label{lemma:flattening-pair}
Let $\mathcal{J}$ and $\mathcal{J}^\ast$ be two codimension $p$ involutive manifolds and $\Sigma$ a symplectic manifold of dimension $2n-2p$ so that
\[
\mathcal{J} \cap \mathcal{J}^\ast = \Sigma
\]
is a transverse intersection. Then there exist Darboux coordinates $\kappa$ so that locally $\mathcal{J}$, (resp. $\mathcal{J}^\ast$, $\Sigma$) coincides with
\begin{align*}
\{ (x,x', \xi, 0)\ &|\ x,\xi\in\C^{n-p},\ x'\in \C^p\}, \\
\{ (x,0, \xi, \xi')\ &|\ x,\xi\in\C^{n-p},\ \xi'\in \C^p\}, \\
\{ (x,0, \xi, 0)\ &|\ x,\xi\in\C^{n-p}\}.
\end{align*}

If additionally $\Sigma=\widetilde{\Sigma\cap\R^{2n}}$, we can assume that $\kappa(\Sigma)\cap \R^{2n} = \kappa(\Sigma\cap\R^{2n})$. 
\end{lemma}

\begin{proof}
This seems like a basic result from symplectic geometry, but we could not find it in the litterature so we give some details. 

We start by flattening $\mathcal{J}$ and $\Sigma$ using the previous results, and notice then that by transversality, $\mathcal{J}$, (resp. $\mathcal{J}^\ast$, $\Sigma$) locally coincide with
\begin{align*}
\{ (x,x', \xi, 0)\ &|\ x,\xi\in\C^{n-p},\ x'\in \C^p\}, \\
\{ (x,f(x,\xi,\xi'), \xi, \xi')\ &|\ x,\xi\in\C^{n-p},\ \xi'\in \C^p\}, \\
\{ (x,0, \xi, 0)\ &|\ x,\xi\in\C^{n-p}\}.
\end{align*}
Here, $f$ is a smooth function satisfying $f(\bullet,\bullet,0) = 0$. Notice that we can do this flattening mapping real points of $\Sigma$ to real points in the case that $\Sigma = \widetilde{\Sigma\cap \R^{2n}}$. We observe that the maps
\[
\zeta_j := x'_j - f_j
\]
vanish on $\mathcal{J}^\ast$. Since it is involutive, we deduce that
\[
\{ \zeta_j, \zeta_k\}_{|\mathcal{J}^\ast} = 0. 
\]
However these Lie brackets do not depend on $x'$ so that they vanish everywhere. We also notice that
\[
\{ \zeta_j, \xi'_k \} = \delta_{jk}.
\]
This suggests to consider the hamilton vector fields $H_{\zeta_j}$ and $H_{\xi'_k}$, $j,k=1\dots p$. From the commutation relations, we see that these vector fields pairwise commute, and define thus an orbit foliation $\mathcal{G}$ by leaves of dimension $2p$. Additionally, along each leaf of this foliation, the $\zeta_j$'s with the $\xi'_k$'s define a system of coordinates.

We also observe that $\Sigma$ is transverse to this foliation, so we denote by $\pi=(g_1, \dots, g_{n-p}, g^\ast_1, \dots, g^\ast_{n-p})$ the local projection to $\Sigma$ along $\mathcal{G}$. On $\Sigma$, $(g,g^\ast)$ coincide with $(x,\xi)$. The family
\[
g_1,\dots, g_{n-p} ,\zeta_1, \dots, \zeta_p\ ;\ g^\ast_1,\dots,g^\ast_{n-p}, \xi'_1, \dots, \xi'_p
\]
defines a local coordinate system. In this coordinate system, by construction, we have flattened our three manifolds as announced; it remains to prove that they are indeed Darboux coordinates.

By construction, in these coordinates
\[
H_{\zeta_j} = \partial_{\xi'_j},\text{ and } H_{\xi'_k} = - \partial_{\zeta_k}.
\]
It follows that
\[
\omega=\sum_j d\zeta_j \wedge d\xi'_k + \omega'(\zeta,\xi',g,g^\ast;dg,dg^\ast). 
\]
Since $\omega$ is closed, $\omega'$ cannot depend on $(\zeta,\xi')$. However since $(g,g^\ast)_{|\Sigma}$ is already a Darboux system, we get
\[
\omega = dg\wedge dg^\ast + d\zeta\wedge d\xi'
\]
\end{proof}
Again, the only local invariants of pairs of involutive manifolds intersecting transversally are dimensions. Motivated by \S \ref{sec:study-phase}, we set
\begin{definition}
Let $\mathcal{J}$ and $\mathcal{J}^\ast$ be involutive manifolds in transverse intersection at $\Sigma$, so that $\Sigma = \widetilde{\Sigma\cap\R^{2n}}$. Also assume that the null foliations $\mathcal{F}$ and $\mathcal{F}^\ast$ have smooth global leaves. We say that $(\mathcal{J},\mathcal{J}^\ast)$ is an \emph{intersecting pair}.
\end{definition}

We obtain a converse to \eqref{eq:pre-factorization of phi}:
\begin{proposition}
Let $(\mathcal{J}, \mathcal{J}^\ast)$ be an intersecting pair. Let
\[
\Lambda(\mathcal{J},\mathcal{J}^\ast):= \cup_{y\in\Sigma} \mathcal{F}(y)\cup \mathcal{F}^\ast(y).
\]
This is a lagrangian manifold. Seen as the relation of an FIO, 
\[
\Lambda(\mathcal{J},\mathcal{J}^\ast) \circ \Lambda (\mathcal{J},\mathcal{J}^\ast) = \Lambda (\mathcal{J},\mathcal{J}^\ast).
\]
We say that $(\mathcal{J},\mathcal{J}^\ast)$ is (strictly) positive if $\Lambda(\mathcal{J},\mathcal{J}^\ast)$ is itself (strictly) positive.
\end{proposition}

Observe that taking the adjoint of the operator translate at the level of the relation in the operation
\[
\Lambda^\ast = \{ (\overline{\beta}, \overline{\alpha})\ |\ (\alpha,\beta)\in\Lambda\}. 
\]
In particular, it makes sense to say that $\Lambda$ is self-adjoint if $\Lambda^\ast = \Lambda$, and we observe that $\Lambda(\mathcal{J},\mathcal{J}^\ast)$ is self-adjoint if and only if $\mathcal{J}^\ast = \overline{\mathcal{J}}$. 

\begin{proof}
The proof that $\Lambda:=\Lambda(\mathcal{J},\mathcal{J}^\ast)$ is lagrangian is very similar to the proof of Lemma \ref{lemma:Lagrangian-associated-with-Involutive-manifold}. We focus on the reproducing property. 
\[
\Lambda(\mathcal{J},\mathcal{J}^\ast) \circ \Lambda (\mathcal{J},\mathcal{J}^\ast) = \{ (\alpha,\beta)\ |\ \exists \gamma,\ (\alpha,\gamma) \text{ and } (\gamma,\beta)\in\Lambda\}.
\]
The condition $\exists \gamma,\ (\alpha,\gamma) \text{ and } (\gamma,\beta)\in\Lambda$ implies that there exist $y,y'\in\Sigma$, so that 
\[
\alpha\in\mathcal{F}(y),\ \gamma\in \mathcal{F}^\ast(y)\cap \mathcal{F}(y'),\ \beta \in \mathcal{F}^\ast(y'). 
\]
The transversality condition implies that $y=y'=\gamma$, and we are done. 
\end{proof}

\begin{proposition}
Let $(\mathcal{J},\mathcal{J}^\ast)$ be an intersecting pair as before. Then $\Lambda(\mathcal{J},\mathcal{J}^\ast)$ is strictly positive and $\Lambda_\R = \Delta_{\Sigma_\R}$ if and only if both $\mathcal{J}$ and $\overline{\mathcal{J}^\ast}$ are strictly positive.
\end{proposition}
In particular, $(\mathcal{J},\overline{\mathcal{J}})$ yields a strictly positive self-adjoint lagrangian $\Lambda(\mathcal{J},\overline{\mathcal{J}})$. 

\begin{proof}
At $\Sigma$, we find
\[
T\Lambda = T\mathcal{F} \oplus T\mathcal{F}^\ast \oplus T(\Lambda_\R),
\]
and $T\mathcal{F}^\ast \subset T\mathcal{F}^\perp$, so the computation of $\omega(v,\overline{v})$ for $v\in T\Lambda$ reduces to the same computation for $\mathcal{F}$ and $\mathcal{F}^\ast$, with the caveat that the symplectic form to be used is $d\xi \wedge dx - d\xi'\wedge dx'$.
\end{proof}

In the local reductions above, we have used symplectomorphisms that preserve the real points of $\Sigma$, but which do not preserve real points in general. If we have an involutive manifold $\mathcal{J}$, and a real-symplectic $\kappa$, then certainly $\kappa(\overline{\mathcal{J}})= \overline{\kappa(\mathcal{J})}$. In particular, using real symplectic maps, we must preserve the class of lagrangians of the form $\Lambda(\mathcal{J},\overline{\mathcal{J}})$. However according to Proposition \ref{prop:normal-form-positive-involutive} these look all locally the same modulo real symplectic maps.

\end{document}